\documentclass[a4paper,10pt]{article}
\usepackage{amsmath,amssymb,amsthm,authblk}

\usepackage[utf8]{inputenc}
\usepackage[T1]{fontenc}
\usepackage{fancyhdr}
\usepackage{textcomp}
\usepackage[dvipsnames]{xcolor}
\usepackage{graphicx, epstopdf}
\usepackage{braket,amsfonts}
\usepackage{dsfont}
\usepackage{mathtools}
\usepackage{mwe}
\usepackage[top=2.4cm,bottom=2.4cm,left=2.4cm,right=2.4cm]{geometry}
\usepackage{enumitem}

\usepackage{hyperref}

\hypersetup{colorlinks=true,linkcolor=red,filecolor=blue,citecolor=blue,urlcolor=teal}

\usepackage[sort,nocompress]{cite}
 
\newcommand{\uvect}{\mathbf{u}}

\newcommand{\m}{\mathbf{m}} 

\DeclareMathOperator{\tr}{\textnormal{tr}}

\newcommand\Ss{\mathbb{S}}

\newcommand{\boldphi}{\boldsymbol{\phi}}

\renewcommand{\rho}{\varrho}

\newcommand{\R}{\mathbb{R}}
\newcommand{\Z}{\mathbb{Z}}
\newcommand{\N}{\mathbb{N}}
\newcommand{\vareps}{\varepsilon}
\newcommand{\ueps}{\uvect_{\varepsilon}}
\newcommand{\tueps}{\widetilde{\uvect}_{\varepsilon}}
\newcommand{\qeps}{q_{\varepsilon}}
\newcommand{\tqeps}{\tilde{q}_{\varepsilon}}
\newcommand{\fxe}{\frac{x}{\varepsilon}}
\newcommand{\Beps}{\mathcal{B}_{\varepsilon}}
\newcommand{\tBeps}{\widetilde{\mathcal{B}}_{\varepsilon}}

\newcommand{\feps}{f_{\varepsilon}}
\newcommand{\tfeps}{\tilde{f}_{\varepsilon}}
\newcommand{\veps}{v_{\varepsilon}}
\newcommand{\weps}{w_{\varepsilon}}
\newcommand{\rhoeps}{\rho_{\varepsilon}}
\newcommand{\trhoeps}{\widetilde{\rho}_{\varepsilon}}
\newcommand{\oeps}{\Omega_{\varepsilon}}

\newcommand{\geps}{\Gamma_{\varepsilon}}
\newcommand{\phieps}{\boldphi_{\varepsilon}}
\newcommand{\psieps}{\psi_{\varepsilon}}
\newcommand{\bftx}{\left(t,x,\frac{x}{\vareps}\right)}
\newcommand{\per}{\mathrm{per}}

\renewcommand{\div}{\mathrm{div}}
\newcommand{\B}{\mathcal{B}}
\newcommand{\Reps}{R_{\varepsilon}}

\newcommand{\Myf}{\mathcal{M}_{Y_f}}

\newcommand{\teps}{\mathcal{T}_{\vareps}}
\newcommand{\foe}{\frac{1}{\varepsilon}}

\newcommand{\ratsw}[1]{\xrightharpoonup[]{2,#1}}
\newcommand{\ratss}[1]{\xrightarrow[]{2,#1}}

\DeclareMathOperator{\weak}{\textnormal{weak}}

\newtheorem{theorem}{Theorem}[section] 
\newtheorem{proposition}[theorem]{Proposition}
\newtheorem{lemma}[theorem]{Lemma} 
\newtheorem{corollary}[theorem]{Corollary} 
\newtheorem{remark}[theorem]{Remark} 
\newtheorem{definition}[theorem]{Definition}

\title{Homogenization of the compressible Navier-Stokes equations \\ via two-scale convergence in perforated domains}

\author[1]{Markus Gahn
	\thanks{E-mail Id: \href{mailto: markus.gahn@uni-a.de}{markus.gahn@uni-a.de}.}}
\author[2]{Kuntal Bhandari \thanks{E-mail Id: \href{mailto: bhandari@math.cas.cz}{bhandari@math.cas.cz}.}}
\author[2]{\v{S}\'{a}rka Ne\v{c}asov\'{a} \thanks{E-mail Id: \href{mailto: matus@math.cas.cz}{matus@math.cas.cz} (corresponding author).}}
\author[3]{Maria Neuss-Radu \thanks{E-mail Id: \href{maria.neuss-radu@math.fau.de}{maria.neuss-radu@math.fau.de}.}}

\affil[1]{{\normalsize Augsburg University, Institute of Mathematics, Universit\"{a}tsstra{\ss}e 14,
86159 Augsburg, Germany} \vspace{.1cm} }
\affil[2]{{\normalsize Institute of Mathematics of the  Czech Academy of Sciences, \v{Z}itn\'{a} 25, 11000 Praha 1, Czech Republic} \vspace{.1cm}}
\affil[3]{{\normalsize Friedrich-Alexander-Universit{\"{a}}t Erlangen-N{\"{u}}rnberg, Cauerstr. 11, 91058 Erlangen, Germany}}

\date{}

\begin{document}

\maketitle

\begin{abstract}

We study the homogenization of the compressible isentropic Navier-Stokes equations in periodically perforated domains where the size of the obstacles is of the same order as the distance between neighboring obstacles. Using the two-scale convergence method, which can be characterized via the unfolding operator, we derive the corresponding macroscopic model determined by Darcy's law. In particular, the macroscopic density satisfies the porous medium equation. The main challenge lies in identifying the pressure term in the limit. We overcome this by establishing the strong two-scale convergence of the densities, which is achieved by controlling the oscillation defect measure of the unfolded densities. A crucial contribution of our work is the development of a methodological framework applicable to more complex compressible fluid models. Furthermore, regarding conservative forces, we extend existing results from the literature to adiabatic constants $\gamma > \frac95$.
  
\end{abstract}

\vspace{.4cm}
	
	
\vspace{.2cm}	

\section{Introduction}

Homogenization problems in fluid mechanics address the study of the asymptotic behavior of fluid flows in domains containing a large number of holes or obstacles. These domains are commonly referred to as perforated domains, with porous media being a typical example. The aim is to describe the effective behavior of the flow as the number of obstacles tends to infinity while their characteristic size simultaneously tends to zero, leading to the emergence of macroscopic models that capture the influence of the underlying microstructure. 

The rigorous mathematical study of the homogenization   process for stationary incompressible fluids started from the work by Tartar \cite{tartar1980incompressible}.
A systematic study of stationary Stokes and Navier-Stokes flows in perforated domains was carried out by  Allaire in \cite{Allaire-ARMA-1990-I, Allaire-ARMA-1990-II}. He considered fluid flow in domains perforated by a family of holes of varying sizes. More precisely, the holes are assumed to have diameter of order $O(\vareps^\alpha)$, where $\vareps$ denotes the typical distance between them. 
In the three-dimensional  setting, he showed that the asymptotic behavior of the fluid depends critically on the value of $\alpha$.
When, $\alpha<3$,
 the effective behavior is governed by Darcy’s law. In the critical case $\alpha=3$, the limit yields the Brinkman law. Finally, when 
$\alpha>3$ the perforations have no influence in the limit, and the homogenized system remains the same as the original Stokes or 
Navier-Stokes equations. The homogenization results have also been extended to more complex fluid models. Some examples for this would include evolutionary incompressible Navier-Stokes equations, investigated, for instance, in the works of Mikelić \cite{Mikelic-1991-ANPA}, Feireisl et al. \cite{Feireisl-Necasova-Namyeyeva},  Lu and Wang \cite{Lu-yang-JMFM}, or studies addressing more general (random) particle configurations and inhomogeneous Dirichlet boundary conditions: Giunti and H\"ofer \cite{Hofer-Giunti}, Hillairet et al. \cite{Hillairet-Sueur-Ayman}, H\"ofer and Jansen \cite{Hofer-Jonas}, very large Reynolds numbers: H\"ofer  \cite{Hofer-Nonlinearity}, and non-Newtonian fluids: Lu and Qian \cite{Lu-Qian-JDE}.

Turning to the case of compressible fluids, homogenization problems for the stationary compressible Navier-Stokes equations in three spatial dimensions were studied by Feireisl and Lu \cite{Feireisl-Lu-JMFM-3D}, and in the two-dimensional setting by Nečasová and Pan \cite{Necasova-Pan-M2AS}. For the non-stationary compressible Navier-Stokes equations, such problems are studied by Lu and Schwarzacher \cite{Lu-Schwarzacher}, Ne\v{c}asov\'{a} and Oschmann \cite{Necasova-Oschmann-CVPDE} and  Oschmann and Pokorn\'{y} \cite{Pokorny-Oschmann-JDE}, while the case of randomly distributed holes is addressed in \cite{Oschmann-BElla-ARMA} by Bella and Oschmann.  Furthermore, homogenization problems for heat-conducting compressible fluids, namely the full 
Navier-Stokes-Fourier system, have been analyzed in the works of Lu and Pokorn\'{y} \cite{Lu-Pokorny-JDE}, Oschmann \cite{Oschman-JMFM} and, Pokorn\'{y} and  Skříšovský \cite{Pokorny-Emil-JEPE}.
 In all these works on compressible fluids, the analysis is carried out in the regime of tiny holes, for which the limiting equations remain unchanged.

In contrast, for the case of large perforations or obstacles  whose size is proportional to their mutual distance, the first homogenization result was achieved by Masmoudi \cite{masmoudi2002homogenization}. In this regime, the limiting equations differ from the original ones, with Darcy’s law emerging as the effective model and the Darcy pressure being described by a porous medium equation. This result was later extended to the Navier-Stokes-Fourier system by Feireisl et al. \cite{feireisl2010homogenization}. 
Very recently, a quantitative homogenization of the compressible Navier-Stokes equations was studied in  \cite{hofer2025quantitative} by H\"ofer et al, considering a domain in the same setting as in \cite{masmoudi2002homogenization}. 

The present paper is also devoted to studying  the homogenization of a viscous compressible isentropic flow, described by the compressible Navier-Stokes equations, in a  periodically perforated domain, where the  diameter
of the obstacles is of the same order as the distance between neighboring obstacles. Our approach is based on the so-called   \textit{two-scale convergence} method  which was initially introduced by Nguetseng \cite{Nguetseng} and later  developed by Allaire \cite{Allaire_TwoScaleKonvergenz}. 
This method provides a rigorous framework for justifying homogenization results, as it captures both macroscopic and microscopic scales simultaneously. It has been successfully applied in numerous studies on homogenization.
A characterization of the two-scale convergence can be given via the unfolding operator,   first introduced by Vogt \cite{VogtHomogenization} and later further investigated by Arbogast et al. \cite{ArbogastDouglasHornung} and Bourgeat et al. \cite{BourgeatLuckhausMikelic} (using the name dilation operator). In more detail, this operator was later studied under the name unfolding operator for example in \cite{CioranescuDamlamianDonatoGrisoZakiUnfolding,Cioranescu_Unfolding2} by Cioranescu and coauthors.  We also refer to the monograph \cite{CioranescuGrisoDamlamian2018} (also by  Cioranescu et al.) for a detailed overview on this topic. More precisely, the unfolding operator maps functions defined on a perforated ($\vareps$-dependent) domain to functions defined on a fixed domain, and depending on a macroscopic and a microscopic variable. The dependence on two variables, where the unfolded sequence is a piecewise constant function with respect to the macroscopic variable, is the price to pay when dealing with the unfolded sequence. The primary advantage of the unfolding operator in its suitability for dealing with nonlinear problems, such as in our case the compressible Navier-Stokes equations, where the strong two-scale convergence of the density is required. This convergence is equivalent to the strong $L^p$-convergence of the corresponding unfolded sequence.

To the best of our knowledge, the homogenization of the compressible Navier-Stokes equations via two-scale convergence method (or equivalently via the unfolding method) has not yet been addressed in the literature. In this work, we employ this method to derive the corresponding macroscopic model, where the microscopic domain is assumed to be periodically perforated (and the size of the obstacles is of the same order as  their mutual distance). 
The contribution of this paper is twofold: First, we provide a general framework based on two-scale convergence for the homogenization of the compressible Navier-Stokes equations. This approach is designed to accommodate more complex settings, such as coupling with auxiliary processes or more complex geometries. To this end, we offer a rigorous, detailed proof of each step in the homogenization process. Second, we extend existing results from the literature \cite{masmoudi2002homogenization, feireisl2010homogenization} to adiabatic constants $\gamma > \frac{9}{5}$. We note, however, that obtaining $\vareps$-uniform a priori bounds on the fluid velocity requires restricting the momentum equation to conservative forces. For more general forces, we further restrict our analysis to the case $\gamma \geq 2$, consistent with \cite{feireisl2010homogenization}.

We start our homogenization process with uniform a priori estimates, which are based on the energy inequality for a solution of the compressible Navier-Stokes equation. From these estimates, weak compactness results in the two-scale sense for the fluid velocity, the fluid pressure and the density are obtained directly from two-scale compactness results. These convergence results allow to pass to the limit in the linear terms in our system. The crucial point now is the identification of the adiabatic pressure term in the limit, for which we need strong two-scale convergence of the density or equivalently  strong $L^p$-convergence for the associated unfolded sequence. For this, we control the so-called oscillation defect measure for the unfolded density. The oscillation defect measure was used in \cite{feireisl2001compactness} by Feireisl (see also the paper by Jiang and Zhang \cite{jiang2001spherically}) to show strong compactness for solutions of the isentropic compressible 
Navier-Stokes equations, and in \cite{feireisl2001existence} by Feireisl et al.  to establish the existence of weak solutions for this  system in the case $\gamma >\frac32$. In the context of homogenization, it was used in \cite{feireisl2010homogenization} to the density extended by its local mean in every micro-cell. Here, we estimate the oscillation defect measure of the unfolded density, for which we in particular use an improved pressure estimate. Again, we present the proof in great detail, to make this method applicable to more complex problems. Based on this strong two-scale convergence for the density to a macroscopic density $\rho_0$, we are able to pass to the limit $\vareps \to 0$ in the micro-model. We obtain on the set $\{\rho_0>0\}$ an incompressible Stokes-cell problem for the limit fluid velocity. Finally, this allows to derive via the continuity equation  a porous medium equation for the macroscopic density $\rho_0$.

\vspace{.3cm}
\noindent 
{\bf Notations.} Let us introduce some basic notations   which will be used throughout this paper. 
For $m\in \N$, $p\in [1,\infty]$ and $U\subset \R^m$, we denote the usual Lebesgue- and Sobolev-spaces by $L^p(U)$ and $W^{1,p}(U)$, respectively.  For $p=2$ we shortly write $H^1(U) := W^{1,2}(U)$.  For Sobolev functions with zero traces on $\partial U$ (for a Lipschitz-domain $U$), we use the notation $W^{1,p}_0(U)$ and $H_0^1(U)$ for $p=2$. If $U$ is bounded, we define 
\begin{align*}
    L_0^2(U):= \left\{ \phi \in L^2(U) \, : \, \int_{U} \phi dx = 0 \right\}.
\end{align*}
For a Banach space $X$, we denote the cartesian product by $X^n$ for $n\in \N$, and for its norm we usually skip the exponent. For example, we write $\|\cdot\|_{L^2(U)} := \|\cdot\|_{L^2(U)^n}$. Further, we denote the usual Bochner-spaces by $L^p((0,T),X)$, and by $W^{1,p}((0,T),X)$ the space of functions in $L^p((0,T),X)$ with generalized derivatives in $L^p((0,T),X)$. Then, $W^{1,p}_0((0,T),X)$ is the space of functions having zero traces in $t=0$ and $t = T$. Again, we use for $p=2$ the notations $H^1((0,T),X)$ and $H^1_0((0,T),X)$.  Further, the space $C^0([0,T], X_{\mathrm{weak}})$ denotes the set of all continuous functions from $[0,T]\to X$ with respect to the weak topology in $X$.

For $Y= (0,1)^n$ and $p\in [1,\infty]$, we denote by $W^{1,p}_{\per}(Y)$ the space of $Y$-periodic Sobolev functions, defined as the closure of $Y$-periodic functions from $C^{\infty}(\R^n) $ with respect to the norm in $W^{1,p}(Y)$. For $p=2$, we write $H_{\per}^1(Y)$.

\vspace{.3cm}
\noindent
{\bf Paper organization.}  In Section \ref{sec:micro_model_main_result} the microscopic model and the main result are formulated. In Section \ref{sec:apriori_estimates}, we obtain some important a priori estimates which are relevant for this work. In particular, we obtain an  improved pressure estimate using the  uniform estimates of the Bogovskii operator for perforated domains, and a suitable  pressure decomposition which is crucial  for the two-scale convergence of the pressure term $p(\rhoeps)$. In Section \ref{sec:compactness}, we introduce the two-scale convergence method, provide the weak two-scale convergence results for the microscopic solutions, and finally prove the strong two-scale convergence for the microscopic densities. Thereafter, in Section \ref{sec:derivation_macro_model}, we derive  the  macroscopic model based the two-scale convergence results. Finally, in Appendix \ref{sec:appendix}, we collect several auxiliary results including some uniform estimates of Bogovskii operator with respect to the perforation parameter.

\section{The microscopic model and the main result}
\label{sec:micro_model_main_result}

In this section we introduce the microscopic geometry, formulate the microscopic equation and give the definition of a weak solution. Further, we formulate the main result including the macroscopic model.

\subsection{Microscopic model}

We consider a periodically perforated domain $\oeps \subset \Omega \subset \R^3$ with $\Omega = (a,b)$ and $a,b\in \Z^3$ such that $a_i < b_i $ for $i=1,2,3$ (basically $a=(a_1,a_2,a_3)$ and $b=(b_1,b_2,b_3)$).  Further, let $0 <\vareps \ll 1$ with $\vareps^{-1} \in \N$. Let $Y_s \subset Y:= (0,1)^3$ strictly included. We define $\Gamma:= \partial Y_s$  and $Y_f := Y \setminus \overline{Y_s}$. Now, we define the perforated domain $\oeps$ and the surface $\geps$ by
\begin{align*}
\oeps &:= \mathrm{int} \left\{\bigcup_{k\in K_{\vareps}} \vareps (\overline{Y_f} + k) \right\} = \Omega \setminus \bigcup_{k \in \Z^3} \vareps (\overline{Y_s} + k),
\\
\geps &:= \bigcup_{k\in K_{\vareps}} \vareps ( \Gamma + k) = \partial \oeps \setminus \partial \Omega,
\end{align*}
where the set $K_{\vareps} \subset \Z^3$ is defined by
\begin{align*}
K_{\vareps}:= \left\{ k \in \Z^3 \, : \, \vareps(Y + k) \subset \Omega \right\}.
\end{align*}
Now, we consider the following compressible Navier-Stokes equations: 
\begin{subequations}\label{micro_model}
\begin{align}
\partial_t \rhoeps + \div(\rhoeps \ueps) &= 0 &\mbox{ in }& (0,T)\times \oeps,
\\
\label{eq:momentum_equation_micro_model}
\vareps^{\lambda}\left( \partial_t (\rhoeps \ueps )  + \div(\rhoeps \ueps \otimes \ueps) \right)- \vareps^2 \div(\Ss(\nabla \ueps)) + \nabla p(\rhoeps) &= {\rhoeps \nabla F} &\mbox{ in }& (0,T)\times \oeps,
\\
\ueps &= 0 &\mbox{ on }& (0,T)\times \partial\oeps,
\\
\rhoeps(0) &=\rho_{\vareps,0} &\mbox{ in }& \oeps,
\\
(\rhoeps \ueps)(0) &= \m_{\vareps,0} &\mbox{ in }&\oeps,
\end{align}
\end{subequations}
with $\lambda > 1 + \frac{3}{\gamma}$, and $F$ is some suitable potential which is smooth enough.  
Here, we define for $s \geq 0$ the pressure term  as 
\begin{align*}
p(s):= a s^{\gamma}
\end{align*}
with $a>0$ and $\gamma >\frac95$, and for a matrix $A\in \R^{3\times 3}$ we define $\Ss$ via
\begin{align*}
\Ss(A):= \mu \left( A + A^\top - \frac23 \tr(A) \mathbb I_3 \right) + \eta \tr (A) \mathbb I_3
\end{align*}
with viscosity coefficients  $\mu >0$ and $\eta \geq 0$.

\begin{remark}
The specific choice of the force term $\nabla F$ in the momentum equation $\eqref{eq:momentum_equation_micro_model}$ via a potential $F$ is necessary to obtain uniform energy estimates with respect to $\vareps$, see Proposition \ref{prop:apriori_basic}. For $\gamma \geq 2$ we can replace $\nabla F$ with a general force $\mathbf{f}\in L^{\infty}((0,T)\times \Omega)^3$, see Remark \ref{rem:general_force} for more details.
\end{remark}

\begin{definition}[Weak solutions]\label{Def:weak_sol}
We say that the pair $(\rhoeps ,\ueps)$ is a  finite energy weak solution of the system $\eqref{micro_model}$ if for given  data
\begin{align}\label{Assump-data}
 \rho_{\vareps,0} \geq 0, \ \int_{\Omega_\vareps} \rho^\gamma_{\vareps,0} \, dx  \leq  c , \ \    
\int_{\Omega_\vareps} \vareps^{\lambda} \frac{|\mathbf m_{\vareps, 0}|^2}{\rho_{\vareps,0}} \, dx \leq c ,    \  \mathbf m_{\vareps,0}=0 \ \text{if} \ \rho_{\vareps,0}=0  , \ \ F\in W^{1,\infty}((0,T)\times \Omega_\vareps) , 
\end{align}
we have the following: 
\begin{itemize}
\item It holds that
\begin{align*}
 &\rhoeps \geq 0 \, \text{ a.e. in } \, (0,T)\times \oeps, \ \   \rhoeps\in C^0([0,T], L^\gamma_{\weak}(\oeps)), \\
&\ueps\in L^2((0,T), W^{1,2}_0(\oeps)), \  \   \rhoeps\ueps\in C^0([0,T],  L^{\frac{2\gamma}{\gamma+1}}_{\weak}(\oeps ) ) .
\end{align*}

\item For any $0 \le \tau \le T$ and any $\psieps \in C_0^{\infty}([0,T)\times \overline{\oeps})$ it holds that
\begin{align}\label{eq:weak_mass_balance}
\int_0^{\tau} \int_{\oeps} \rhoeps \partial_t \psieps + \rhoeps \ueps \cdot \nabla \psieps dx dt = \int_{\oeps} \rhoeps(\tau) \psieps(\tau) dx - \int_{\oeps} \rho_{\vareps,0} \psieps(0) dx.
\end{align}

\item For every $0 \le \tau \le T$ and $\phieps \in C_0^{\infty}([0,T)\times \oeps)^3$,   it holds that
\begin{align}\label{eq:weak_impulse}
\int_0^{\tau} \int_{\oeps} \Big[\vareps^{\lambda} \rhoeps \ueps \cdot \partial_t \phieps + \vareps^{\lambda}\rhoeps \ueps \otimes \ueps : \nabla \phieps - \vareps^2 \mathbb{S}(\nabla \ueps) : \nabla \phieps + p(\rhoeps) \div \phieps \notag \\
+ \rhoeps \nabla F \cdot \phieps \Big] dx dt 
= \int_{\oeps} \vareps^{\lambda}(\rhoeps \ueps)(\tau) \phieps(\tau) dx - \int_{\oeps} \vareps^{\lambda} \m_{\vareps,0} \phieps(0) dx.
\end{align}

\item The following  energy inequality holds
\begin{align}\label{eq:energy-ineq}
&\int_{\oeps}  \left( \frac12 \vareps^\lambda  \rhoeps |\ueps|^2 + \frac{p(\rhoeps)}{\gamma - 1}  \right)(\tau) dx  + \int_0^\tau \int_{\oeps} \vareps^2  \Ss(\nabla \ueps) : \nabla \ueps dx dt  
\notag \\
&\le \int_{\oeps} \left( \vareps^{\lambda}\frac{ |\m_{\vareps,0}|^2}{2 \rho_{\vareps,0}} + \frac{p(\rho_{\vareps,0})}{\gamma -1} \right)dx + \int_0^\tau \int_{\oeps} \rhoeps \nabla F \cdot \ueps dx dt 
\end{align}
for a.a.  $0 \le \tau \le T$. 

\item Finally, the zero extension $(\trhoeps,\tueps)$ satisfies 
\begin{align}\label{eq:renormalized_continuity_equation}
\partial_t b(\trhoeps) + \div(b(\trhoeps) \tueps) + (\trhoeps b'(\trhoeps) - b(\trhoeps)) \div \tueps &= 0 
 \ \text{ in } \ \mathcal{D}'((0,T)\times \R^3), 
\end{align}
for any $b \in C^0([0,\infty)) \cap C^1((0,\infty))$ with the growth condition $|b'(z)z|\le C (z^{\omega} + z^{\frac{\gamma}{2}})$ for $z>0$ and $\omega \in (0, \frac{\gamma}{2})$.
\end{itemize}
\end{definition}

\subsection{Main result}

We are now in the position to state the main result of this paper. 

\begin{theorem}
Let $\gamma>\frac{9}{5}$ and $\lambda > 1 + \frac{3}{\gamma}$. Furthermore, let  $(\rhoeps, \ueps)$ be a weak global-in-time  solution to the microscopic model \eqref{micro_model}. Then, the sequence $(\rhoeps, \ueps)_{\vareps}$ converges in the two-scale sense to a limit function $(\rho_0,\uvect_0)$ with $\rho_0\in L^\infty((0,T),L^\gamma(\Omega \times Y_f))\cap L^{\gamma+\nu}((0,T)\times \Omega \times Y_f)$ (with $\nu=\frac{2}{3}\gamma-1$) and $\uvect_0\in L^2((0,T)\times \Omega, H^1(Y))^3$ with $\uvect_0=0$ in $Y_s$, and there exists $p_1\in L^2((0,T),L^q(\Omega,L^2(Y_f)))$ with $q=\min\{2,\gamma\}$ such that:  $(\uvect_0, p_1)$ satisfies   \begin{align*} 
    -2\mu \div_y (D_y( \uvect_0)) + \nabla_y p_1 &= \rho_0\nabla F -  \nabla 
    p_0 &\mbox{ in }& \{\rho_0 > 0\}  \times Y_f,
\\
\div_y( \uvect_0) &= 0 &\mbox{ in }&  \{\rho_0 > 0\}  \times Y_f,
\\
\uvect_0 &= 0 &\mbox{ on }& \{\rho_0 > 0\}  \times \Gamma,
\\
(\uvect_0,p_1) \mbox{ are } Y\mbox{-periodic} ,
\end{align*}
with $p_0=a\rho_0^\gamma$, and $\rho_0$ satisfies 
\begin{align}
\begin{aligned}\label{eq:porous_media_rho0}
|Y_f|\partial_t \rho_0 + \div_x \left(K \rho_0^2 \nabla F - a\rho_0\nabla \rho_0^{\gamma}\right) = 0 & \qquad  \mbox{in } (0,T)\times \Omega, \\
\rho_0(0) = |Y_f| \rho_{0,0} & \qquad \mbox{in } \Omega. 
\end{aligned}
\end{align}
\end{theorem}

In summary, the macroscopic behavior of the system is completely described by the porous medium equation $\eqref{eq:porous_media_rho0}$ for the macroscopic density $\rho_0$. The fluid velocity $\uvect_0$, which approximates the microscopic fluid velocity $\ueps$ in the two-scale sense, as well as the pressure $p_1$, can be obtained through $\rho_0$, $F$ and suitable cell problems, see $\eqref{eq:representation_u0}$ and $\eqref{eq:representation_p1}$. In particular, the Darcy velocity $ \bar{\uvect}_0(t,x):=\int_{Y_f} \uvect_0(t,x,y) dy $ is given by 
\begin{align*}
    \bar{\uvect}(t,x)  = K\left[\rho_0 \nabla F - \nabla_x p_0 \right](t,x)
\end{align*}
with the permeability tensor $K\in \R^{3\times 3}$ defined in $\eqref{def:permeability_tensor}$.

\section{A priori estimates}
\label{sec:apriori_estimates}
Here we derive uniform a priori estimates for a weak solution $(\rhoeps,\ueps)$ of $\eqref{micro_model}$ (see also \cite{feireisl2010homogenization,hofer2025quantitative,masmoudi2002homogenization} for similar estimates). We start with the basic energy inequality:

\begin{proposition}\label{prop:apriori_basic}
It holds that
\begin{align*}
\vareps^{\lambda}\left\|\rhoeps |\ueps|^2\right\|_{L^{\infty}((0,T), L^1(\oeps))} + \|\ueps\|_{L^2((0,T)\times \oeps)} + \vareps \|\nabla \ueps \|_{L^2((0,T)\times \oeps)} + \|\rhoeps\|_{L^{\infty}((0,T), L^{\gamma}(\oeps))} \le C
\end{align*}
for a constant $C>0$ independent of $\vareps$.

\end{proposition}

\begin{proof}
The proof follows from the energy inequality \eqref{eq:energy-ineq} and the choices of given data \eqref{Assump-data} together with the Korn inequality. Since this procedure is quite standard, we just  estimate the source term on the right-hand side of \eqref{eq:energy-ineq}, to emphasize the specific choice of the source term via a potential. 
Indeed, we have 
\begin{align*}
\int_0^\tau \int_{\oeps} \rho_\vareps
\nabla F \cdot \uvect_\vareps dx dt &=
- \int_0^\tau \int_{\oeps} F \div (\rhoeps \uvect_\vareps) dx dt = 
 \int_0^\tau \int_{\oeps} F \partial_t \rhoeps  dx dt \\
 & = - \int_0^\tau \int_{\oeps} \rho_\vareps \partial_t F dxdt + \Big[\int_{\oeps} \rho_\vareps F dx  \Big]_{t=0}^\tau  \ ,
\end{align*}
and thus 
\begin{align*}
\bigg|  \int_0^\tau \int_{\oeps} \rho_\vareps
\nabla F \cdot \uvect_\vareps dx dt  \bigg|  \leq C ,
\end{align*}
using the mass conservation law and the fact that $F\in W^{1,\infty}((0,T)\times \Omega_\vareps)$.  
\end{proof}

\begin{remark}\label{rem:general_force} 
If we replace $\nabla F$ by a general force $\mathbf{f} \in L^{\infty}((0,T)\times \Omega)^3$, we obtain (see also \cite{hofer2025quantitative}) with the Korn inequality
\begin{align}
\int_{\oeps} \rhoeps \mathbf{f} \cdot \uvect_{\vareps} dx \le C \| \rhoeps\|_{L^2(\oeps)}^2 + \frac12 \int_{\oeps} \vareps^2 \mathbb{S}(\nabla \uvect_{\vareps}) : \nabla \uvect_{\vareps} dx.
\end{align}
Hence, for $\gamma \geq 2$ the above term can be controlled via the Gr\"onwall inequality.
\end{remark}

For the strong (two-scale) convergence of the density we derive an improved estimate of the pressure. Such an idea was already used in \cite{masmoudi2002homogenization} (see also \cite{feireisl2010homogenization}). More precisely, in \cite{masmoudi2002homogenization} the restriction operator, first introduced by Tartar in  \cite{tartar1980incompressible}, was used to extend the pressure to the whole domain. Our ideas are very similar, but we will not construct an extension (this can be done, but is not necessary for the proof of the main result). In particular, we use the Bogovskii operator instead of the restriction operator (of course, these concepts are somehow equivalent).
We decompose the pressure $p(\rhoeps)$ in a remaining part of order $1$ (not vanishing for $\vareps\to 0$), and something tending to $0$ in a suitable sense for $\vareps \to 0$. For this, our estimates are similar to the estimates in \cite{hofer2025quantitative}. However, in contrast to 
\cite[Lemma 2.4]{hofer2025quantitative} we formulate in Proposition \ref{prop:improved_pressure_estimate_Bogovskii} the estimate in such a way that the right-hand side includes the Bogovskii operator (not only the test-function for the pressure). The advantage is to use the properties of the Bogovskii operator and special choices of test functions. More precisely, terms including time-derivatives lead to  terms of the form $\Beps(\div g_{\vareps})$ with $g_{\vareps}\cdot \nu = 0$ on $\geps$. Such terms can be controlled in a suitable way by using the estimates for the Bogovskii operator, see Section \ref{sec:Bogovskii_operator} in the appendix.
\\

Before we formulate the main result of this section in Proposition \ref{prop:improved_pressure_estimate_Bogovskii}, we introduce the following functional setting  (see also  \cite[Remark 3.3]{hofer2025quantitative}). For Banach spaces $X_{\vareps}^1$ and $X_{\vareps}^2$ embedded in an ambient topological vector space, we use the notation $\qeps \in X_{\vareps}^1 + \vareps X_{\vareps}^2$ if there exist $q_{\vareps}^1 \in X_{\vareps}^1$ and $q_{\vareps}^2 \in X_{\vareps}^2$ such that $\qeps = q_{\vareps}^1 + q_{\vareps}^2$. Further, we say that $\qeps $ is bounded in  $X_{\vareps}^1 + \vareps X_{\vareps}^2$, if 
\begin{align*}
    \|\qeps^1\|_{X_{\vareps}^1} \le C, \qquad \|\qeps^2\|_{X_{\vareps}^2} \le C \vareps
\end{align*}
for a constant $C>0$ independent of $\vareps$. This notation can be generalized in an obvious way to an arbitrary finite number of Banach spaces. 
\\

Now, we introduce a well-known duality concept. For the sake of completeness, we describe this procedure in a detailed way. For Banach spaces $X$ and $Y$ continuously embedded into an ambient Hausdorff space we have $(X\cap Y)' = X' + Y'$. More precisely, for $z' \in (X\cap Y)'$ with 
\begin{align*}
    \langle z', z\rangle \le C_X \|z\|_X + C_Y \|z\|_Y \qquad \mbox{ for all } z \in X\cap Y,
\end{align*}
there exist $x' \in X'$ and $y' \in Y'$, such that for all $z \in X\cap Y$ it holds that
\begin{align*}
    \langle z',z\rangle = \langle x', z\rangle  + \langle y' ,y\rangle,
    \\
    \|x'\|_X \le C_X, \qquad \|y'\|_Y \le C_Y , 
\end{align*}
see also \cite[Lemma 3.5]{hofer2025quantitative}.

In the following, we have to deal particularly with functions spaces of the form $X = L^s((0,T),B)$ with $s \in [1,\infty]$ and a Banach space $B$. Hence, we summarize some standard representations for functionals in the space $X'$. We consider the case $B:= W^{1,p}(\oeps)'$ for $p \in (1,\infty)$, hence the space $W^{1,p}(\oeps)$ (and its dual) is reflexive and therefore, we have for $s \in [1,\infty)$ and its dual exponent $s'$
\begin{align*}
L^s((0,T),W^{1,p}(\oeps)')' \cong L^{s'}((0,T),W^{1,p}(\oeps)^{\prime\prime}) \cong L^{s'}((0,T),W^{1,p}(\oeps)).
\end{align*}
More precisely, let us denote by $J$ the isometry between $W^{1,p}(\oeps)$ and $W^{1,p}(\oeps)^{\prime\prime}$. Then, for $G_{\vareps} \in L^s((0,T),W^{1,p}(\oeps)')'$ we define $g_{\vareps} := J^{-1} G_{\vareps} \in L^{s'}((0,T),W^{1,p}(\oeps))$. Now, for $\psieps \in L^s((0,T),L^q(\oeps))$ with $q= \frac{pn}{p(n +1 ) - n}$ for $p<n$ and $q>1$ for $p=n$ and $q=1$ for $p>n$ (this is valid for arbitrary dimension $n\in \N$), we use the canonical embedding of $L^s((0,T),L^q(\oeps))  $ into $L^s((0,T),W^{1,p}(\oeps)^{\prime} )$ 
\begin{align*}
\langle  G_{\vareps} , \psieps \rangle_{L^{s'}((0,T),W^{1,p}(\oeps)^{\prime\prime}), L^s((0,T),W^{1,p}(\oeps){\prime})} &= \langle \psieps ,g_{\vareps} \rangle_{L^s((0,T),W^{1,p}(\oeps){\prime}), L^{s'}((0,T),W^{1,p}(\oeps))}
\\
&= \int_0^T \int_{\oeps} g_{\vareps} \psieps dx dt.
\end{align*}

\vspace{.3cm}

Now, we are able to formulate the following improved pressure estimate:
\begin{proposition}\label{prop:improved_pressure_estimate_Bogovskii}
There exists $q_{\vareps} \in L^1((0,T)\times \oeps)$ and bounded in 
\begin{align*}
 L^2((0,T),H^1(\oeps)) + L^{\infty}((0,T),W^{1,\gamma}(\oeps)) + \vareps\left( L^2((0,T)\times \oeps) + L^{\infty}((0,T),L^{\gamma}(\oeps)) \right)
\end{align*}
such that for all
\begin{align*}
\psieps \in W^{1,1}((0,T),L^{\frac{2\gamma}{\gamma - 1}}(\oeps)) \cap L^{\infty}((0,T),L^{\frac{s'\gamma p}{(p-2)\gamma - p}}(\oeps))
\end{align*}
having mean-value zero 
for $p \in \left( \frac{2\gamma}{\gamma - 1},\min\{6,\frac{4\gamma}{(\gamma-1)(2-s)}\}\right]$ with $s \in (1,2]$ arbitrary  and $\theta =\frac{3p-6}{2p} $ it holds that 
\begin{align}\label{ineq:improved_pressure_Bogovskii}
&\bigg| \int_0^T \int_{\oeps} (p(\rhoeps) - q_{\vareps} ) \psieps dx dt \bigg|
    \le  C \bigg\{\vareps^{\frac{\lambda}{2}} \|\Beps(\partial_t \psieps)\|_{L^1((0,T),L^{\frac{2\gamma}{\gamma -1}}(\oeps))}  \notag \\ 
  & \qquad \qquad   +  \vareps^{\frac{\lambda}{s} - \frac{2\theta}{s} -1} \vareps \|\nabla \Beps (\psieps)\|_{L^{\infty}((0,T),L^{\frac{s'\gamma p}{(p-2)\gamma - p}}(\oeps))}
    + \vareps^{\frac{\lambda}{2}} \|\Beps(\psieps)\|_{W^{1,1}((0,T),L^{\frac{2\gamma}{\gamma -1}}(\oeps))} \bigg\} , 
\end{align}
where $\mathcal B_{\vareps}$ is the Bogovskii operator as introduced in Appendix \ref{sec:Bogovskii_operator}.
\end{proposition}

\begin{proof}
Let $\psieps$ be as in the statement. Due to the mean-value zero condition we have with the Bogovskii operator $\Beps$ (see Appendix  \ref{sec:Bogovskii_operator}) that $\div \Beps \psieps = \psieps$ leading to 
\begin{align*}
\int_0^{T} \int_{\oeps} p(\rhoeps) \psieps dx dt =& -\int_0^{T} \int_{\oeps} \vareps^{\lambda} \rhoeps \ueps \cdot \Beps(\partial_t \psieps) dx dt - \int_0^T \int_{\oeps} \vareps^{\lambda} \rhoeps \ueps \otimes \ueps : \nabla \Beps (\psieps) dx dt
\\
&+ \int_0^{T} \int_{\oeps} \vareps^2 \mathbb{S}(\nabla \ueps) : \nabla \Beps(\psieps) dx dt - \int_0^{T} \int_{\oeps}\rhoeps \nabla F \cdot \Beps (\psieps) dx dt 
\\
&+ \int_{\oeps}  \Big( \vareps^{\lambda} (\rhoeps \ueps)(T) \Beps(\psieps)(T)  - \vareps^{\lambda} \m_{\vareps,0} \Beps(\psieps)(0) \Big) dx
\\
=&: \sum_{i=1}^5 \langle Q_{\vareps,i} ,\psieps\rangle_{Y'_{\vareps,i},Y_{\vareps,i}} =: \sum_{i=1}^5 I_{\vareps}^i.
\end{align*}
with $Q_{\vareps,i} \in Y_{\vareps,i}^{\prime}$ and  the function spaces $Y_{\vareps,i}$ defined by 
\begin{align*}
    Y_{\vareps,1} &:= W^{1,1}((0,T),L^{\frac{2\gamma}{\gamma -1}}(\oeps))   ,  
    \\
    Y_{\vareps,2} &:= L^{s'}((0,T), L^{\frac{s\gamma p}{(p-2)\gamma -p}}(\oeps)),
    \\
    Y_{\vareps,3} &:= L^2((0,T)\times \oeps),
    \\
    Y_{\vareps,4} &:= L^1((0,T),L^{\frac{\gamma}{\gamma -1}}(\oeps)),
    \\
    Y_{\vareps,5} &:= W^{1,1}((0,T),L^{\frac{2\gamma}{\gamma -1}}(\oeps)).
\end{align*}
That the operators $Q_{\vareps,i}$ are in fact bounded functionals on the respective spaces $Y_{\vareps,i}$ is obtained from the estimates below. We will see that  $Q_{\vareps,3}$ and $Q_{\vareps,4}$ play a special role, because their norm is of order $1$, while the other norms scale with an order of $\vareps^{\kappa}$ with $\kappa >0$. Let us estimate the terms $I_{\vareps}^i$ separately. Here, we show two types of estimates. First, we show a bound with respect to the Bogovskii operator and then a bound with respect to $\psieps$ with respect to suitable norms.  

For $I_{\vareps}^1$, we have with the a priori estimates from Proposition \ref{prop:apriori_basic} and Corollary \ref{cor:Bogovskii}. Indeed, we compute 
\begin{align*}
|I_{\vareps}^1| &\le \vareps^{\lambda}\|\sqrt{\rhoeps}\|_{L^{\infty}((0,T),L^{2\gamma}(\oeps))} \|\sqrt{\rhoeps} \ueps\|_{L^{\infty}((0,T),L^2(\oeps))} \|\Beps(\partial_t \psieps)\|_{L^1((0,T),L^{\frac{2\gamma}{\gamma -1}}(\oeps))} 
\\
& \leq C \vareps^{\frac{\lambda}{2}} \|\partial_t \psi_\vareps\|_{L^1((0,T); L^{\frac{2\gamma}{\gamma-1}}(\Omega_\vareps))} .
\end{align*}
The last inequality shows in particular that $Q_{\vareps,1}$ is an element of the dual space of $W^{1,1}((0,T),L^{\frac{2\gamma}{\gamma -1}} (\oeps))$.

For the second term  $I_{\vareps}^2$, we use the H\"older inequality with  $ \frac{1}{s'} + \frac{1}{\gamma s} + \frac{2}{p s} + \frac{1}{r} = 1$  
for $p \in \left( \frac{2\gamma}{\gamma - 1} , 6\right]$ (in particular $p>2$) and $r = \frac{s\gamma p}{(p-2)\gamma - p} \in (2,\infty)$ in the spatial variable and $\frac{1}{\infty} + \frac{1}{\infty} + \frac{1}{s} + \frac{1}{s'} = 1$ to obtain 
\begin{align*}
|I_{\vareps}^2| &\le \vareps^\lambda \left\| \rhoeps |\ueps|^2 \right\|_{L^\infty((0,T),L^1(\oeps))}^{\frac{1}{s'}} \left\| \rhoeps \right\|_{L^\infty((0,T),L^\gamma(\oeps))}^{\frac{1}{s}} \left\| \ueps \right\|_{L^2((0,T),L^p(\oeps))}^{\frac{2}{s}} \|\nabla \Beps (\psieps) \|_{L^{s'}((0,T),L^r(\oeps))}
\\
&\le C \vareps^{\lambda - \frac{\lambda}{s'}} \left\| \ueps \right\|_{L^2((0,T),L^p(\oeps))}^{\frac{2}{s}} \|\nabla \Beps (\psieps) \|_{L^{s'}((0,T),L^r(\oeps))}.
\end{align*}
For the term including $\ueps$ we use the Gagliardo-Nirenberg inequality to get with $\theta =\frac{3p-6}{2p} $ and the a priori bound for $\ueps $ from Proposition \ref{prop:apriori_basic} (and the Poincar\'e inequality in Lemma \ref{lem:Poincare_inequality} in the appendix) that 
\begin{align}\label{ineq:Gagliardo_Nirenberg_ueps}
\|\ueps\|_{L^2((0,T),L^p(\oeps))}^2 \le C \int_0^T \left(\|\nabla \ueps\|^{2\theta}_{L^2(\oeps)} \|\ueps\|^{2(1-\theta)}_{L^2(\oeps)} + \|\ueps\|_{L^2(\oeps)}^2\right) dt \le C \left(\vareps^{-2\theta} + 1\right).
\end{align}
Hence, we get 
\begin{align*}
|I_{\vareps}^2| &\le  C \vareps^{\frac{\lambda}{s} - \frac{2\theta}{s} - 1} \vareps \|\nabla \Beps(\psieps)\|_{L^{s'}((0,T),L^r(\oeps))} \notag \\
    &\le C \vareps^{\frac{\lambda}{s} - \frac{2\theta}{s} - 1} \left( \|\psieps\|_{L^{s'}((0,T),W^{1,r'}(\oeps)')} + \vareps \|\psieps\|_{L^{s'}((0,T),L^r(\oeps))} \right).
\end{align*}
From the above inequality,  observe that we  need $\frac{\lambda}{s} - \frac{2\theta}{s} - 1 >0$ (if we want the right-hand side to vanish for $\vareps \to 0$). Since  $p>\frac{2\gamma}{\gamma -1}$, we obtain the condition $\lambda > 1 + \frac{3}{\gamma}$ for $s$ close to $1$ and $p $ close to $\frac{2\gamma}{\gamma - 1}$.

For $I_{\vareps}^3$ we obtain, again by using Proposition \ref{prop:apriori_basic}, that 
\begin{align}
\begin{aligned}\label{ineq:I_eps_3}
|I_{\vareps}^3| &\le C \vareps^2 \|\nabla \ueps \|_{L^2((0,T)\times \oeps)} \|\nabla \Beps (\psieps)\|_{L^2((0,T)\times \oeps)} 
\\
&\le C \vareps \|\nabla \Beps(\psieps) \|_{L^2((0,T)\times \oeps)} 
\\
&\le C \left(\|\psieps \|_{L^2((0,T),H^1(\oeps)')} + \vareps \|\psieps \|_{L^2((0,T) \times \oeps)}\right).
\end{aligned}
\end{align}
We see that the first term on the right-hand side in the last line has no scaling factor of the form $\vareps^{\varkappa}$ with $\varkappa >0$. Hence, we expect that this term has a contribution for $\vareps \to 0$ in the limit. Therefore, we will discuss below the functional $Q_{\vareps,3}$ in more detail and decompose it in a suitable way.

Next, we estimate $I_{\vareps}^4$. We have
\begin{align*}
|I_{\vareps}^4| &\le \|\nabla F\|_{L^{\infty}((0,T)\times \oeps)} \|\rhoeps\|_{L^{\infty}((0,T),L^{\gamma }(\oeps))} \|\Beps (\psieps)\|_{L^1((0,T),L^{\frac{\gamma }{\gamma  -1}}(\oeps))}
\\
&\le C\vareps \|\nabla \Beps(\psieps)\|_{L^1((0,T), L^{\frac{\gamma }{\gamma  -1}}(\oeps))}
\\
&\le C \left( \|\psieps\|_{L^1((0,T),W^{1,\gamma}(\oeps)')} + \vareps \|\psieps\|_{L^1((0,T),L^{\frac{\gamma }{\gamma  -1}}(\oeps))} \right), 
\end{align*}
by using the Poincar\'{e} inequality from Lemma \ref{lem:Poincare_inequality} and the estimate of 
Bogovskii operator given by Lemma \ref{lem:Bogovskii_Beps}. 
Again, as in the previous case $i=3$, we have a term of order $1$, and we will also decompose $Q_{\vareps,4}$ below.

Finally, we calculate the term  $I_{\vareps}^5$.  First, observe that  under the assumption on the initial data in Definition \ref{Def:weak_sol}, 
\begin{align*}
\|\m_{\vareps,0}\|_{L^{\frac{2\gamma}{\gamma +1}}(\oeps)} \le C\vareps^{-\frac{\lambda}{2}} . 
\end{align*}
This, together with the H\"older inequality for $\frac{1}{2\gamma} + \frac12 + \frac{1}{s} = 1$ with $s = \frac{2\gamma}{\gamma -1}$, we compute 
\begin{align*}
|I_{\vareps}^5| \le C \vareps^{\lambda} \left( \|\rhoeps\|_{L^{\infty}((0,T),L^{\gamma}(\oeps))}^{\frac12} \|\rhoeps |\ueps|^2\|_{L^{\infty}((0,T),L^1(\oeps))}^{\frac12} + \vareps^{-\frac{\lambda}{2}} \right) \|\Beps (\psieps)\|_{L^{\infty}((0,T),L^{\frac{2\gamma}{\gamma -1}}(\oeps))}.
\end{align*}
Using the continuity of the embedding $W^{1,1}((0,T),L^{\frac{2\gamma}{\gamma -1}}(\oeps))$ into $C^0([0,T],L^{\frac{2\gamma}{\gamma -1}}(\oeps))$,  we obtain with the 
a priori estimates from Proposition \ref{prop:apriori_basic}
\begin{align}\label{ineq:_I_eps_5}
|I_{\vareps}^5| &\le C \vareps^{\frac{\lambda}{2}} \left( \|\Beps(\psieps)\|_{L^1((0,T),L^{\frac{2\gamma}{\gamma -1}}(\oeps))} + \| \Beps (\partial_t \psieps)\|_{L^1((0,T),L^{\frac{2\gamma}{\gamma -1}}(\oeps))} \right)
\notag \\
&\le C \vareps^{\frac{\lambda}{2}} \left( \|\psieps\|_{W^{1,1}((0,T),W^{1,\frac{2\gamma}{\gamma + 1}}(\oeps)')} + \vareps \|\psieps \|_{W^{1,1}((0,T),L^{\frac{2\gamma}{\gamma -1}}(\oeps))}\right).
\end{align}

Next, we decompose the functionals $Q_{\vareps,3}$ and $Q_{\vareps,4}$.  From inequality $\eqref{ineq:I_eps_3}$ we obtain
\begin{align*}
|\langle Q_{\vareps,3} , \psieps \rangle_{Y_{\vareps,3}}| \le \left(\|\psieps \|_{L^2((0,T),H^1(\oeps)')} + \vareps \|\psieps \|_{L^2((0,T) \times \oeps)}\right)
\end{align*}
for all $\psieps \in Y_{\vareps,3} = L^2((0,T)\times \oeps)$. Hence, there exist 
 $q_{\vareps,3}^1 \in L^2((0,T),H^1(\oeps))$ and $q_{\vareps,3}^2 \in L^2((0,T)\times \oeps)$ with 
\begin{align*}
\|q_{\vareps,3}^1\|_{L^2((0,T),H^1(\oeps))} + \vareps^{-1} \|q_{\vareps,3}^2\|_{L^2((0,T)\times \oeps)} \le C,
\end{align*}
such that 
\begin{align*}
\langle Q_{\vareps,3} , \psieps \rangle_{Y_{\vareps,3}} = \int_0^T \int_{\oeps } q_{\vareps,3} \psieps dx dt,
\end{align*}
where $q_{\vareps,3} := q_{\vareps,3}^1 + q_{\vareps,3}^2$. 

In the same way we can treat $Q_{\vareps,4}$ and obtain $q_{\vareps,4}^1 \in L^{\infty}((0,T),W^{1,\gamma}(\oeps))$, $q_{\vareps,4}^2 \in L^{\infty}((0,T),L^{\gamma}(\oeps))$ with
\begin{align*}
    \|q_{\vareps,4}^1 \|_{L^{\infty}((0,T),W^{1,\gamma}(\oeps))} + \vareps^{-1} \|q_{\vareps,4}^2 \|_{L^{\infty}((0,T),L^{\gamma}(\oeps))} \le C,
\end{align*}
such that for all $\psieps \in L^1((0,T),L^{\frac{\gamma}{\gamma - 1}}(\oeps))$
\begin{align*}
\langle Q_{\vareps,4} , \psieps \rangle_{Y_{\vareps,4}} = \int_0^T \int_{\oeps } q_{\vareps,4} \psieps dx dt,
\end{align*}
where $q_{\vareps,4}:= q_{\vareps,4}^1 + q_{\vareps,4}^2$. 
Altogether, we obtain with $q_{\vareps} := q_{\vareps,3} + q_{\vareps,4}$ 
\begin{align*}
    \bigg| \int_0^T \int_{\oeps}& (p(\rhoeps) - q_{\vareps} ) \psieps dx dt \bigg| \le |I_{\vareps}^1 + I_{\vareps}^2 + I_{\vareps}^5|
    \\
    \le& C \bigg\{\vareps^{\frac{\lambda}{2}} \|\Beps(\partial_t \psieps)\|_{L^1((0,T),L^{\frac{2\gamma}{\gamma -1}}(\oeps))}   +  \vareps^{\lambda - 2\theta -1} \vareps \|\nabla \Beps (\psieps)\|_{L^{\infty}((0,T),L^{\frac{\gamma p}{(p-2)\gamma - p}}(\oeps))}
    \\
    &+ \vareps^{\frac{\lambda}{2}} \|\Beps(\psieps)\|_{W^{1,1}((0,T),L^{\frac{2\gamma}{\gamma -1}}(\oeps))} \bigg\}.
\end{align*}
\end{proof}

\begin{remark}\label{rem:improved_pressure_estimate}\
\begin{enumerate}
[label = (\roman*)]
\item In the estimate above we see that we need 
\begin{align*}
\psieps \in & W^{1,1}((0,T),L^{\frac{2\gamma}{\gamma - 1}}(\oeps)) \cap L^{\gamma}((0,T),L^{\frac{\gamma p}{(p-2)\gamma - p}}(\oeps)) \cap L^2((0,T),L^2(\oeps)) 
\\
&\cap L^1((0,T),L^{\frac{\gamma }{\gamma  -1}}(\oeps))  .
\end{align*}
However, the choice in the Proposition is enough since we have for $p  \in \left( \frac{2\gamma}{\gamma - 1},\min\{6,\frac{4\gamma}{(\gamma-1)(2 - s)}\}\right]$
\begin{align*}
 \frac{s\gamma p }{(p-2)\gamma - p} \geq \max\left\{2 , \frac{\gamma }{\gamma  -1}\right\}.
\end{align*}

\item The function $\theta(p) := \theta = \frac{3p -6}{2p} $ is increasing and positive for $p>2$ (we have $p> \frac{2\gamma}{\gamma -1}$). Hence, we have $\theta > \frac{3}{2\gamma}$. Further, from the assumption $\lambda > 1 + \frac{3}{\gamma}$ we get $\frac{\lambda}{s} - \frac{2\theta}{s} - 1>0$ if we choose  $s$ close to $1$ and $p $ close to $\frac{2\gamma}{\gamma - 1}$.

\item In \cite{hofer2025quantitative} the term $I_{\vareps}^2$ was estimated for $s = 1$. This leads to a bound for the test-function including the Bogovskii operator in $L^{\infty}$ in time, which causes difficulties for the strong two-scale convergence of the density. Therefore, we somehow interpolated between the bound for the density and the kinetic energy. 

\item The result can be modified to functions without mean value zero. However, for our purpose this is not necessary and we omit it.

\item In \cite[Lemma 3.2]{hofer2025quantitative} a decomposition of the pressure $p(\rhoeps)$ is derived. The estimates in the proof of Proposition \ref{prop:improved_pressure_estimate_Bogovskii} allow a similar decomposition. For our proof of the strong (two-scale) convergence of the density it is more appropriate to work with the stronger estimate $\eqref{ineq:improved_pressure_Bogovskii}$, where we can use the properties of the Bogovskii operator $\Beps$.

\end{enumerate}

\end{remark}

Let us finally formulate in the following Corollary the precise decomposition of $\qeps$ from Proposition \ref{prop:improved_pressure_estimate_Bogovskii}. This quantity will be crucial for the compactness of $p(\rhoeps)$, since their two-scale limits coincide.
\begin{corollary}\label{cor:decomposition_pressure}
The function $q_{\vareps} $ from Proposition \ref{prop:improved_pressure_estimate_Bogovskii} can be decomposed into
\begin{align*}
    \qeps  = \sum_{i=1}^4 q_{\vareps}^i \in  L^2((0,T),H^1(\oeps)) + L^{\infty}((0,T),W^{1,\gamma}(\oeps)) + \vareps\left( L^2((0,T)\times \oeps) + L^{\infty}((0,T),L^{\gamma}(\oeps)) \right),
\end{align*}
such that
\begin{align*}
\|\qeps^1\|_{L^2((0,T),H^1(\oeps))} + \|\qeps^2\|_{L^{\infty}((0,T),W^{1,\gamma}(\oeps))} + \vareps^{-1} \|\qeps^3\|_{L^2((0,T)\times \oeps)} + \vareps^{-1} \|\qeps^4\|_{L^{\infty}((0,T),L^{\gamma}(\oeps))} \le C
\end{align*}
for a constant $C>0$ independent of $\vareps$.
\end{corollary}

As in the proof of Proposition \ref{prop:improved_pressure_estimate_Bogovskii}, we obtain the following estimate for the full pressure for more specific test-functions.
\begin{lemma}\label{lem:refined_pressure_estimate}
For every $\psieps$ with 
\begin{align*}
\psieps \in H^1((0,T),L^{\frac{6\gamma}{5\gamma - 6}}(\oeps)) \cap L^2((0,T),L^{\frac{3\gamma}{2\gamma -3}}(\oeps))
\end{align*}
having mean value zero and $\psi \in \mathcal{D}(0,T)$ it holds that
\begin{align*}
\int_0^T \psi(t) \int_{\oeps} p(\rhoeps) \psieps dx dt \le& C \bigg( \vareps^{\lambda - 1} \|\psi\|_{L^{\infty}(0,T)} \| \Beps(\partial_t \psieps)\|_{L^2((0,T),L^{\frac{6\gamma}{5\gamma - 6}}(\oeps))} 
\\
&+ \vareps^{\lambda - 2} \|\psi\|_{L^{\infty}(0,T)}  \|\nabla \Beps (\psieps)\|_{L^{\infty}((0,T),L^{\frac{3\gamma}{2\gamma -3}}(\oeps))}
\\
&+\vareps \|\psi\|_{L^{\infty}(0,T)} \|\nabla \Beps(\psieps) \|_{L^2((0,T)\times \oeps)} + \|\psi\|_{L^{\infty}(0,T)} \|\Beps (\psieps)\|_{L^1((0,T),L^{\frac{\gamma }{\gamma  -1}}(\oeps))}
\\
&+ \vareps^{\frac{\lambda}{2}} \|\psi'\|_{L^1(0,T)} \|\Beps(\psieps)\|_{L^{\infty}((0,T),L^{\frac{2\gamma}{\gamma -1}}(\oeps))} 
\bigg).
\end{align*}
\end{lemma}
\begin{proof}
We follow the arguments in the proof of Proposition \ref{prop:improved_pressure_estimate_Bogovskii}, but use slightly different norms and test-functions with compact support in $(0,T)$. 
We have
\begin{align*}
&\int_0^{T} \psi(t) \int_{\oeps} p(\rhoeps) \psieps dx dt \\
=& -\int_0^{T} \psi(t)\int_{\oeps} \vareps^{\lambda} \rhoeps \ueps \cdot \Beps(\partial_t \psieps) dx dt - \int_0^T \psi(t) \int_{\oeps} \vareps^{\lambda} \rhoeps \ueps \otimes \ueps : \nabla \Beps (\psieps) dx dt
\\
&+ \int_0^{T} \psi(t)\int_{\oeps} \vareps^2 \mathbb{S}(\nabla \ueps) : \nabla \Beps(\psieps) dx dt - \int_0^{T} \psi(t) \int_{\oeps}\rhoeps \nabla F \cdot \Beps (\psieps) dx dt 
\\
&+ \int_0^T \psi'(t) \int_{\oeps} \vareps^{\lambda} \rhoeps \ueps \cdot \Beps(\psieps)dx dt 
\\
 =&: \sum_{i=1}^5 I_{\vareps}^i.
\end{align*}
The main difference, compared to the proof of Proposition \ref{prop:improved_pressure_estimate_Bogovskii}, is the last term $I_{\vareps}^5$. The other terms include just the additional factor $\psi(t)$ (we can consider $\psi \psieps$ instead of $\psieps$ in the previous proof). Now, we estimate $I_{\vareps}^1$ in the following way:
\begin{align*}
|I_{\vareps}^1 | &\le \vareps^{\lambda}\|\psi\|_{L^{\infty}(0,T)} \|\rhoeps\|_{L^{\infty}((0,T),L^{\gamma}(\oeps)} \|\ueps\|_{L^2((0,T),L^6(\oeps))} \| \Beps(\partial_t \psieps)\|_{L^2((0,T),L^{\frac{6\gamma}{5\gamma - 6}}(\oeps))} 
\\
&\le C\vareps^{\lambda - 1} \|\psi\|_{L^{\infty}(0,T)} \| \Beps(\partial_t \psieps)\|_{L^2((0,T),L^{\frac{6\gamma}{5\gamma - 6}}(\oeps))}.
\end{align*}
The term $I_{\vareps}^2$ is  estimated by
\begin{align*}
|I_{\vareps}^2| \le& \vareps^{\lambda}  \|\psi\|_{L^{\infty}(0,T)}\|\rhoeps\|_{L^{\infty}((0,T),L^{\gamma}(\oeps))} \|\ueps\|^2_{L^2((0,T),L^6(\oeps))} \|\nabla \Beps(\psieps)\|_{L^{\infty}((0,T),L^{\frac{3\gamma}{2\gamma-3}}(\oeps))}
\\
\le& C \vareps^{\lambda - 2} \|\psi\|_{L^{\infty}(0,T)} \|\nabla \Beps(\psieps)\|_{L^{\infty}((0,T),L^{\frac{3\gamma}{2\gamma-3}}(\oeps))}.
\end{align*}
The terms $I_{\vareps}^3$ and $I_{\vareps}^4$ can be treated as in the proof of Proposition \ref{prop:improved_pressure_estimate_Bogovskii}. Finally, for $I_{\vareps}^5$ we have (similar to $I_{\vareps}^1$ in the proof of Proposition \ref{prop:improved_pressure_estimate_Bogovskii})
\begin{align*}
|I_{\vareps}^5| &\le  \vareps^{\lambda}\|\psi'\|_{L^1(0,T)}\|\sqrt{\rhoeps}\|_{L^{\infty}((0,T),L^{2\gamma}(\oeps))} \|\sqrt{\rhoeps} \ueps\|_{L^{\infty}((0,T),L^2(\oeps))} \|\Beps(\psieps)\|_{L^{\infty}((0,T),L^{\frac{2\gamma}{\gamma -1}}(\oeps))} 
\\
&\le C \vareps^{\frac{\lambda}{2}} \|\psi'\|_{L^1(0,T)} \|\Beps(\psieps)\|_{L^{\infty}((0,T),L^{\frac{2\gamma}{\gamma -1}}(\oeps))} .
\end{align*}
This gives the desired result.
\end{proof}

As a consequence, we can derive the following  uniform refined pressure estimate:
\begin{proposition}\label{prop:apriori_pressure_refined}
There exists $0<\nu$ such that 
\begin{align*}
\int_0^T \int_{\oeps} \rhoeps^{\gamma + \nu} dx dt \le C
\end{align*}
for a constant $C>0$ independent of $\vareps$. In particular, it holds that
\begin{align*}
\|p(\rhoeps)\|_{L^{\frac{\gamma + \nu}{\gamma}}((0,T)\times \oeps)} \le C . 
\end{align*}
In particular, for $\gamma\le 3$ we can choose $\nu = \frac23 \gamma - 1$.
\end{proposition}

\begin{proof}
As  test-functions in Lemma \ref{lem:refined_pressure_estimate} we choose
\begin{align*}
\psieps(t,x) := \left(\rhoeps^{\nu} - \frac{1}{|\oeps|} \int_{\oeps} \rhoeps^{\nu} dx \right)
\end{align*}
and $\psi \in \mathcal{D}(0,T)$.
We may proceed exactly as in the proof of Proposition \ref{prop:strong_compactness_rhoeps} later, where we will use inequality \eqref{ineq:aux_error_pressure_strong_conv_density} instead. Since the argument is entirely analogous, and in fact even simpler than the one given for Proposition \ref{prop:strong_compactness_rhoeps}, we omit the full proof here and will present the detailed steps later.  Additional terms here are (for the sake of simplicity we give here a formal argument and ignore the regularization in time)
\begin{align*}
\vareps \|\nabla &\Beps(\psieps) \|_{L^2((0,T)\times \oeps)} +  \|\Beps (\psieps)\|_{L^1((0,T),L^{\frac{\gamma }{\gamma  -1}}(\oeps))}
\\
&\le C\left(\|\rhoeps^{\nu}\|_{L^2((0,T)\times \oeps)} + \|\rhoeps^{\nu}\|_{L^1((0,T),L^{\frac{\gamma}{\gamma -1}}(\oeps))}\ \right) \le C\
\end{align*}
for $\nu \le \gamma -1 $ and $\nu \le \frac{\gamma}{2}$. Further, we easily obtain (for $\nu \le \gamma$)
\begin{align*}
   \left| \int_0^T \psi \int_{\oeps} p(\rhoeps) \int_{\oeps} \rhoeps^{\nu} dz dx \right| \le  C \|\psi\|_{L^{\infty}(0,T)}.
\end{align*}
Next, we have  
\begin{align*}
\vareps^{\lambda -2} \|\nabla \Beps(\psieps)\|_{L^{\infty}((0,T),L^{\frac{3\gamma}{2\gamma -3}}(\oeps))} \le C \vareps^{\lambda -3} \|\rhoeps^{\nu}\|_{L^{\infty}((0,T),L^{\frac{3\gamma}{2\gamma -3}}(\oeps))} \le C \vareps^{\lambda -3}
\end{align*}
for $\nu \le \frac23 \gamma -1$. Further, we have
\begin{align*}
\vareps^{\frac{\lambda}{2}} \|\Beps(\psieps)\|_{L^{\infty}((0,T),L^{\frac{2\gamma}{\gamma -1}}(\oeps))} \le C \vareps^{\frac{\lambda}{2}} \|\rhoeps^{\nu}\|_{L^{\infty}((0,T),L^{\frac{2\gamma}{\gamma -1}}(\oeps))} \le C
\end{align*}
for $\nu \le \frac12\gamma - \frac12$. Finally, we have  formally (again neglecting the regularization in time) 
\begin{align*}
    \partial_t \psieps = -\div(\rhoeps^{\nu} \ueps) + (1-\nu) \Big[\rhoeps^{\nu} \div \ueps - \frac{1}{|\oeps|} \int_{\oeps} \rhoeps^{\nu} \div \ueps dx \Big] 
\end{align*}
and therefore  with Corollary \ref{cor:Bogovskii}
\begin{align*}
&\vareps^{\lambda -1} \|\Beps (\partial_t \psieps)\|_{L^2((0,T),L^{\frac{6\gamma}{5\gamma - 6}}(\oeps))} \\
\le & C \vareps^{\lambda -1} \left( \|\Beps(\div(\rhoeps^{\nu}\ueps ))\| _{L^2((0,T),L^{\frac{6\gamma}{5\gamma - 6}}(\oeps))} + \bigg\|\Beps\Big[\rhoeps^{\nu} \div \ueps -\frac{1}{|\oeps|} \int_{\oeps} \rhoeps^{\nu} \div \ueps dx \Big]    \bigg\|_{L^2((0,T),L^{\frac{6\gamma}{5\gamma - 6}}(\oeps))} \right)
\\
\le & C \vareps^{\lambda -1} \left(\|\rhoeps^{\nu}\ueps \| _{L^2((0,T),L^{\frac{6\gamma}{5\gamma - 6}}(\oeps))} + \vareps^{-1} \|\rhoeps^{\nu} \div \ueps \|_{L^2((0,T),L^{\frac{6\gamma}{7\gamma -6}}(\oeps))} \right)
\\
\le & C \vareps^{\lambda -1} \|\rhoeps^{\nu}\|_{L^{\infty}((0,T),L^{\frac{3\gamma}{2\gamma -3}}(\oeps))} \left( \|\ueps\|_{L^2((0,T),L^6(\oeps))} + \vareps^{-1} \|\div \ueps\|_{L^2((0,T),L^2(\oeps))} \right)
\\
\le & C \vareps^{\lambda -3} \|\rhoeps^{\nu}\|_{L^{\infty}((0,T),L^{\frac{3\gamma}{2\gamma -3}}(\oeps))} \le C\vareps^{\lambda -3}
\end{align*}
for $\nu \le \frac23 \gamma -1$.

Altogether, we obtain (after the regularization process and $\psi \to 1$)
\begin{align*}
\int_0^T \int_{\oeps} \rhoeps^{\gamma + \nu} dx dt \le& C \left( 1  +  \vareps^{\lambda -3}\right).
\end{align*}
This gives the desired result.
\end{proof}

\section{Compactness results}
\label{sec:compactness}

In this section, we derive the compactness results for the microscopic solution, which are necessary to pass to the limit $\vareps \to 0$ and obtain the macroscopic limit problem. As the underlying topology we use the two-scale convergence, see \cite{Allaire_TwoScaleKonvergenz,Nguetseng}, respectively the equivalent definition via the weak/strong convergence of the associated unfolded sequence, see \cite{CioranescuGrisoDamlamian2018}.

\subsection{The two-scale convergence}\label{Sec-Two-scale} 

We start with the definition of the two-scale convergence of a sequence, which was first introduced (in the stationary case) in \cite{Nguetseng} and later further developed in \cite{Allaire_TwoScaleKonvergenz}. We also refer to the overview \cite{LukkassenNguetsengWallTSKonvergenz}.

\begin{definition}
Let $p,s \in [1,\infty)$. We say a sequence $\weps \in L^s((0,T),L^p( \Omega))$ converges (weakly) in the two-scale sense to a limit function $w_0 \in L^s((0,T), L^p( \Omega \times Y))$ if for all $\phi \in L^{s'}((0,T), L^{p'}( \Omega , C_{\per}^0(Y) ) )$ it holds that 
\begin{align*}
\lim_{\vareps \to 0} \int_0^T \int_{\Omega} \weps (t,x) \phi\bftx dx dt = \int_0^T \int_{\Omega} \int_Y w_0(t,x,y) \phi(t,x,y) dy dx dt.
\end{align*}
We write $\weps \ratsw{s,p} w_0$, and $\weps \ratsw{p} w_0$ for $p=s$. If additionally it holds that
\begin{align*}
\lim_{\vareps\to 0} \|\weps\|_{L^s((0,T), L^p( \Omega))} = \|w_0\|_{L^s((0,T), L^p( \Omega \times Y))}
\end{align*}
we say the sequence converges strongly in the two-scale sense and write $\weps \ratss{s,p} w_0$, respectively $\weps \ratss{p} w_0$ for $p=s$.
\end{definition}

We have the following important two-scale compactness results (see \cite{Allaire_TwoScaleKonvergenz,LukkassenNguetsengWallTSKonvergenz}):
\begin{lemma}\label{lem:two_scale_compactness_basic} Let $p \in (1,\infty)$.
\begin{enumerate}
[label = (\roman*)]
\item Every bounded sequence in $L^p((0,T)\times \Omega)$ has a (weakly) two-scale convergent subsequence.

\item Let $\weps \in L^p((0,T),W^{1,p}(\Omega))$ with 
\begin{align*}
\|\weps\|_{L^p((0,T)\times \Omega)} + \vareps \|\nabla \weps \|_{L^p((0,T)\times \Omega)} \le C
\end{align*}
for a constant $C>0$ independent of $\vareps$. Then there exists $w_0 \in L^p((0,T)\times \Omega,W^{1,p}_{\per}(Y))$ such that up to a subsequence
\begin{align*}
\weps \ratsw{p} w_0,\qquad \vareps \nabla \weps \ratsw{p} \nabla_y w_0.
\end{align*}
\end{enumerate}
\end{lemma}

It is well-known, that a product between a weakly and a strongly two-scale convergent sequence converges in the distributional sense to the product of the limit functions.

\subsection{The unfolding operator}\label{sec:unfolding}

We give the definition of the unfolding operator and summarize some important properties. For details, we refer to  \cite{CioranescuGrisoDamlamian2018}.
For $p \in [1,\infty]$ we define the unfolding operator
\begin{align*}
\teps: L^p((0,T)\times \Omega ) \rightarrow L^p((0,T)\times \Omega \times Y), \qquad \teps \weps(t,x,y):= \weps\left( t, \vareps \left[\fxe \right] + \vareps y \right).
\end{align*}
In the same way, we define the unfolding operator (with the same notation) as an operator  $\teps: L^p((0,T)\times \oeps ) \rightarrow L^p((0,T)\times \Omega \times Y_f)$.
We have the following important properties of the unfolding operator, see \cite{CioranescuGrisoDamlamian2018}:
\begin{lemma} Let $p \in [1,\infty)$.
\begin{enumerate}
[label = (\roman*)]
\item For every $\weps \in L^p((0,T)\times \Omega)$ it holds that 
\begin{align*}
\int_0^T \int_{\Omega} \int_{Y} \teps \weps dy dx dt = \int_0^T \int_{\Omega} \weps dx dt.
\end{align*}
In particular, it holds that
\begin{align*}
\| \teps \weps \|_{L^p((0,T)\times \Omega \times Y)} = \|\weps \|_{L^p((0,T)\times \Omega)}.
\end{align*}

\item For $\weps \in L^p((0,T),W^{1,p}(\Omega))$ we have $\teps \weps \in L^p((0,T)\times \Omega,W^{1,p}(Y))$ with $\nabla_y \teps \weps = \vareps \teps \nabla \weps$.
\end{enumerate}
The results are also valid if we replace $\Omega$ with $\oeps$ and $Y$ with $Y_f$.
\end{lemma}

The following results gives the crucial relation between the two-scale convergence and the unfolding operator, see \cite{BourgeatLuckhausMikelic} and \cite{CioranescuGrisoDamlamian2018} for arbitrary $p \in (1,\infty)$ 
\begin{lemma}
Let $p\in (1,\infty)$. A sequence $\weps \in L^p((0,T)\times \Omega)$ converges weakly (strongly) in the two-scale sense to a limit function $w_0 \in L^p((0,T)\times \Omega \times Y)$ if and only if the associated unfolded sequence $\teps \weps$ converges weakly (strongly) to $w_0$ in $L^p((0,T)\times \Omega \times Y)$.

\end{lemma}

Finally, in a slight abuse of notation compared to the standard unfolding literature, we define the mean of a function $\phi \in L^1(Y_f)$ by
\begin{align}\label{def:Myf}
    \Myf(\phi) := \int_{Y_f} \phi dy.
\end{align}

\subsection{Weak two-scale compactness for the microscopic solutions}\label{sec:weak-two-scale}

First, we show some basic weak compactness results for the microscopic quantities $\ueps$, $\rhoeps$ and $p(\rhoeps)$, which follow directly from the uniform a priori bounds from Section \ref{sec:apriori_estimates}.  For the improved regularity of the limit function $p_0$ we use the improved pressure estimate from Proposition \ref{prop:improved_pressure_estimate_Bogovskii} and the  decomposition result in Corollary \ref{cor:decomposition_pressure}.

\begin{proposition}\label{prop:weak_compactness}
There exist $\rho_0 \in L^{\infty}((0,T),L^{\gamma}(\Omega \times Y_f)) \cap L^{\gamma + \nu}((0,T)\times \Omega \times Y_f)$ (with $\nu = \frac23 \gamma -1$ from Proposition \ref{prop:apriori_pressure_refined}), $\uvect_0 \in L^2((0,T)\times \Omega, H_{\per}^1(Y))^3$ with $\uvect_0 = 0$ in $Y_s$, and $p_0 \in L^{\frac{\gamma + \nu}{\gamma}}((0,T)\times \Omega \times Y_f)$ such that up to a subsequence (the function $\ueps$ is extended by zero to the whole domain $\Omega$), we have 
\begin{align*}
\ueps &\ratsw{2} \uvect_0,
\\
\vareps \nabla \ueps &\ratsw{2} \nabla_y \uvect_0,
\\
\chi_{\oeps}\rhoeps &\ratsw{\gamma + \nu} \chi_{Y_f} \rho_0,
\\
\chi_{\oeps} p(\rhoeps) &\ratsw{\frac{\gamma + \nu}{\gamma}} \chi_{Y_f} p_0 .
\end{align*}
Further, we have
\begin{align}\label{weak-two-scale-q_eps}
\chi_{\oeps} q_\vareps   \ratsw{\frac{\gamma + \nu}{\gamma}} \chi_{Y_f} \left( p_0  - \frac{1}{|\Omega|} \int_{\Omega} p_0 dx \right)
\end{align}
and $p_0$ is independent of $y$  and fulfills
$p_0 \in L^2((0,T),H^1(\Omega)) + L^{\infty}((0,T),W^{1,\gamma}(\Omega))$.
\end{proposition}
\begin{proof}
The convergence results for $\ueps$, $\vareps \nabla \ueps$, $\rhoeps$ and $p(\rhoeps)$ with $p_0 \in L^{\frac{\gamma + \nu}{\gamma}}((0,T)\times \Omega \times Y_f)$ are
a direct consequence of the a priori estimates in Proposition \ref{prop:apriori_basic} and \ref{prop:apriori_pressure_refined} together with the two-scale compactness results in Lemma \ref{lem:two_scale_compactness_basic}.
For the convergence of $\qeps$  we first use  the decomposition of $\qeps $ from Corollary \ref{cor:decomposition_pressure}, we immediately obtain
\begin{align*}
    \chi_{\oeps} \qeps^3 &\rightarrow 0 &\mbox{ in }& L^2((0,T)\times \Omega),
    \\
    \chi_{\oeps}\qeps^4 &\rightarrow 0 &\mbox{ in }& L^{\infty}((0,T),L^{\gamma}(\Omega)).
\end{align*}
Let us have a closer look at the critical term $\qeps^1$ and $\qeps^2$. Using well-known extension operators (see e.g., \cite{Acerbi1992}) preserving the uniform bounds of $\qeps^1$ and $\qeps^2$, we obtain bounded  functions $\tqeps^1 \in L^2((0,T),H^1(\Omega))$ and $\tqeps^2 \in L^{\infty}((0,T),L^{\gamma}(\Omega))$, and limit functions $q_0^1\in L^2((0,T),H^1(\Omega))$ and $q_0^2 \in L^{\infty}((0,T),L^{\gamma}(\Omega))$, such that up to a subsequence
\begin{align*}
    \tqeps^1 \ratsw{2} q_0^1, \qquad \tqeps^2 \ratsw{s,\gamma} q_0^2
\end{align*}
for every $s \in [1,\infty)$. Now, we choose $\psi \in C_0^{\infty}((0,T)\times \Omega \times Y_f)$ with mean-value zero with respect to $\Omega \times Y_f$ and define $\psieps (t,x):= \psi\bftx $  and 
\begin{align*}
\widetilde{\psieps}:= \psieps - \frac{1}{|\oeps|} \int_{\oeps} \psieps dx.
\end{align*}
We emphasize that we have (for every $t \in (0,T)$)
\begin{align*}
    \frac{1}{|\oeps|} \int_{\oeps} \psieps dx  \rightarrow  \frac{1}{|\Omega| | Y_f|} \int_{\Omega} \int_{Y_f} \psi dy dx = 0.
\end{align*}
Hence, from the compactness results of $p(\rhoeps)$ and $\qeps$ (the decomposition), we obtain
\begin{align*}
    \lim_{\vareps \to 0} \int_0^T \int_{\oeps} (p(\rhoeps) - \qeps) \widetilde{\psieps} dx dt  =\lim_{\vareps \to 0} \int_0^T \int_{\oeps} (p(\rhoeps) - \qeps) \psieps dx dt = \int_0^T \int_{\Omega} \int_{Y_f} (p_0 - q_0^1 - q_0^2) \psi dy dx dt.
\end{align*}
Further, using the properties of the Bogovskii-operator from Corollary \ref{cor:Bogovskii}, we obtain (for arbitrary $s,q \in [1,\infty)$)
\begin{align*}
 \|\Beps(\widetilde{\psieps})\| _{W^{1,s}((0,T),L^q(\oeps))} + \vareps \|\nabla \Beps(\widetilde{\psieps})\|_{L^{s}((0,T),L^q(\oeps))} \le C,
\end{align*}
with a constant $C>0$ independent of $\vareps$ (but of course depending on $\psi$). Hence, we obtain from inequality $\eqref{ineq:improved_pressure_Bogovskii}$ in Proposition \ref{prop:improved_pressure_estimate_Bogovskii} with a constant $\varkappa >0$ (see also Remark \ref{rem:improved_pressure_estimate})
\begin{align}\label{limit-p(rho)-q}
\left| \int_0^T \int_{\oeps} (p(\rhoeps) - \qeps) \widetilde{\psieps} dx dt \right| \le C\vareps^{\varkappa} \rightarrow 0
\end{align}
for $\vareps \to 0$, which implies $p_0 = q_0^1 + q_0^2$ up to a constant and, in particular, $p_0$ is independent of $y$.
\end{proof}

\subsection{Strong two-scale convergence}\label{sec:strong two-scale}
In the previous section we derived several weak two-scale limits and checked various properties of the limit functions. Of course, we expect $p_0 = a \rho_0^{\gamma}$ (hence $\rho_0$ has to be independent of $y$), but this result is not covered by the weak compactness results. Further, at the moment we are not able to pass to the limit in the weak formulations  $\eqref{eq:weak_mass_balance}$ and  $\eqref{eq:weak_impulse}$, since we need strong two-scale convergence of $\rhoeps$ to pass to the limit in the terms including products of the form $\rhoeps \ueps$. The main result of this section is the following Proposition which gives the strong two-scale compactness for the density $\rhoeps$ and consequently for  the pressure $p(\rhoeps)$.

\begin{proposition}\label{prop:strong_compactness_rhoeps} 
The limit density $\rho_0$ is independent of the microscopic variable $y$, i.e., we have $\rho_0(t,x,y) = \rho_0(t,x)$, and up to a subsequence it holds for $s \in [1 , \gamma + \nu)$, with $\nu = \frac23 \gamma -1$ for $\gamma \le 3$,
\begin{align*}
\teps \rhoeps \rightarrow \rho_0 \quad\mbox{strongly in } L^s((0,T)\times \Omega \times Y_f).
\end{align*}
In other words $\chi_{\oeps}\rhoeps \ratss{s}  \chi_{Y_f}\rho_0$. In particular, we obtain
\begin{align*}
\teps p(\rhoeps) \rightarrow p_0 = a \rho_0^{\gamma} \quad\mbox{strongly in } L^q((0,T)\times \Omega \times Y_f)
\end{align*}
for some $q >1$. In other words, $\chi_{\oeps} p(\rhoeps) \ratss{q} \chi_{Y_f} p_0 = a \rho_0^{\gamma}$. Moreover, we have $p_0 \in L^2((0,T),H^1(\Omega)) + L^{\infty}((0,T),W^{1,\gamma}(\Omega))$.

\end{proposition}

\begin{proof}
The main idea is related to the control of the so-called oscillation defect measure (see \cite{feireisl2017singular}, and \cite{feireisl2010homogenization} in the context of homogenization), of the unfolded sequence $\teps \rhoeps$. However, here we will not use the precise definition, but many aspects which are related to the control of the oscillation defect measure. Several ideas in the proof might be well-known in the theory for compressible fluids. However, for the sake of completeness, we also provide the  details.
We use the following notation. For a sequence $\veps$ in $L^1((0,T)\times \Omega)$ (here, if not stated otherwise, all functions are extended by zero from $\oeps$ to the whole domain $\Omega$) we denote its weak two-scale limit in $L^1$ (if it exists) by $\bar{v}$. More precisely, we have $\teps \veps \rightharpoonup \bar{v}$ in $L^1((0,T)\times \Omega \times Y)$.

We define for $k\in \N$ the cut-off function (see also \cite{feireisl2004dynamics}) 
\begin{align*}
T_k(z) := k T\left(\frac{z}{k}\right) \qquad  z \in \R
\end{align*}
with $T \in C^{\infty}(\R)$ such that
\begin{align*}
T(z):= \begin{cases}
z &\mbox{ for } z \in [0,1],
\\
2 &\mbox{ for } z \geq 3,
\\
-T(-z) &\mbox{ for } z \in (-\infty,0],
\end{cases}
\end{align*}
and $T$ is concave on $[0,\infty)$. We consider the following quantity, which is strongly connected with the 
so-called oscillation defect measure (see e.g. \cite{feireisl2010homogenization})
\begin{align*}
 \int_0^T \int_{\Omega} \int_{Y_f} \left| T_k \left(\teps \rhoeps \right) - T_k(\rho_0) \right|^{\gamma + 1} dx dt.
\end{align*}
First of all, we show that (see also \cite[Section 6.4]{feireisl2004dynamics})
\begin{align}\label{ineq:aux1_strong_conv_density}
a \left| T_k \left(\teps \rhoeps\right) - T_k(\rho_0)\right|^{\gamma + 1} \le p\left(\teps  \rhoeps - \rho_0\right) \left(T_k(\teps \rhoeps )- T_k(\rho_0)\right).
\end{align}
For $p$ we have for all $y \geq z \geq 0$
\begin{align*}
p(y) - p(z) \geq p(y - z).
\end{align*}
Further, from $|T_k(y) - T_k(z)| \le |y-z|$ we obtain
\begin{align*}
p\left(|T_k(y) - T_k(z)|\right) \le p(|y-z|).
\end{align*}
Hence, we obtain
\begin{align*}
a \left| T_k \left(\teps \rhoeps\right) - T_k(\rho_0)\right|^{\gamma + 1} &= p\left(|T_k(\teps \rhoeps) - T_k(\rho_0)|\right) |T_k(\teps \rhoeps - T_k(\rho_0)|
\\
&\le p\left(|\teps \rhoeps - \rho_0|\right) |T_k(\teps \rhoeps) - T_k(\rho_0)|
\\
&\le \left( p(\teps \rhoeps ) - p(\rho_0)\right) (T_k(\teps \rhoeps) - T_k(\rho_0)),
\end{align*}
which is $\eqref{ineq:aux1_strong_conv_density}$.
Hence, we obtain  for every $\psi \in C_0^{\infty}(I') $ with a fixed interval $I'$ such that $\overline{I'} \subset (0,T)$ and $\psi\geq 0$ (later we choose a sequence converging to $1$ in $L^1(0,T)$) 
\begin{align}
\begin{aligned}\label{ineq:aux_estimate_osc_defect}
 a  \int_0^T & \int_{\Omega} \int_{Y_f} \psi \left| T_k \left(\teps \rhoeps\right) - T_k(\rho_0)\right|^{\gamma + 1}  dy dx dt
\\
\le&  \int_0^T \int_{\Omega} \int_{Y_f} \psi \left( p\left(\teps  \rhoeps\right) -  p\left(\rho_0\right) \right) \left(T_k(\teps \rhoeps )- T_k(\rho_0)\right) dy dx dt
\\
=&\int_0^T \psi \int_{\Omega} \int_{Y_f} p(\teps \rhoeps) T_k(\teps \rhoeps) - \overline{p(\rho)} \,\overline{T_k(\rho)} dy dx dt
\\
&- \int_0^T \psi \int_{\Omega} \int_{Y_f} p(\teps \rhoeps) T_k(\rho_0) + p(\rho_0) T_k(\teps \rhoeps) - p(\rho_0) T_k(\rho_0) - \overline{p(\rho)} \,\overline{T_k(\rho)}  dy dx dt
\\
=&: A_{\vareps,k}^1 + A_{\vareps,k}^2.
\end{aligned}
\end{align}
Since $p$ is convex and $T_k$ is concave, we obtain for the second term (see \cite[Theorem 2.11]{feireisl2004dynamics}) for every $k\in \N$
\begin{align}
    \lim_{\vareps \to 0 } A_{\vareps,k}^2 = - \int_0^T \psi \int_{\Omega} \int_{Y_f} \left(\overline{p(\rho)}  - p(\rho_0)\right) \left( T_k(\rho_0) - \overline{T_k(\rho)} \right) dy dx dt  \le 0.
\end{align}
We have to estimate the first term $A_{\vareps,k}^1$. In the following, for an arbitrary function $h \in L^1(U)$ with $U\subset \R^3$ open and bounded we define
\begin{align*}
    \tilde{h}:= h - \frac{1}{|U|} \int_U h dx.
\end{align*}
Now,  we obtain with the properties of the unfolding operator and the fact that $p_0 = \overline{p(\rho)}$  (see $\eqref{def:Myf}$ for the definition of $\Myf$)
\begin{align*}
A_{\vareps,k}^1 =& \int_0^T \int_{\oeps} \psi  p(\rhoeps) T_k(\rhoeps) dx dt - \int_0^T \int_{\Omega}\psi  p_0 \Myf\left(\overline{T_k(\rho)}\right) dx dt 
\\
=& \left\{\int_0^T \int_{\oeps} \psi  \widetilde{p(\rhoeps) } T_k(\rhoeps) dx dt - \int_0^T \int_{\Omega}\psi \widetilde{ p_0} \Myf\left(\overline{T_k(\rho)}\right) dx dt  \right\}
\\
&+ \left\{\frac{1}{|\oeps|} \int_0^T \psi \int_{\oeps} T_k(\rhoeps) dx \int_{\oeps} p(\rhoeps) dz dt  - \frac{1}{|\Omega|} \int_0^T \psi \int_{\Omega} p_0 dx \int_{\Omega} \Myf\left( \overline{T_k(\rho)}\right) dz dt  \right\}
\\
=&: A_{\vareps,k}^{1,1} + A_{\vareps,k}^{1,2} .
\end{align*}
It is easy to check, that for any $k$ we have $A_{\vareps,k}^{1,2} \rightarrow 0$ for $\vareps \to 0$. For the first term $A_{\vareps,1}^{1,1}$ we obtain with a time-regularization $\left[T_k(\rhoeps)\right]_{\delta}$, see \cite[Lemma 6.5]{novotny2004introduction} of $T_k(\rhoeps)$ and $q_{\vareps}$ from Proposition \ref{prop:improved_pressure_estimate_Bogovskii}
\begin{align}
\begin{aligned}\label{ineq:aux_strong_conv_density_Aeps_k_11}
A_{\vareps,k}^{1,1} = & \int_0^T \int_{\oeps} \psi  p(\rhoeps) \widetilde{[T_k(\rhoeps)]_{\delta}}  dx dt + \underbrace{\int_0^T \int_{\oeps}\psi  \widetilde{p(\rhoeps)} \left( T_k(\rhoeps) - \left[T_k(\rhoeps)\right]_{\delta}\right) dx dt }_{=: a_{\vareps,k}^{\delta}}  
\\
& \quad - \int_0^T \int_{\Omega} \psi  \widetilde{p_0} \Myf\left(\overline{T_k(\rho)}\right) dx dt 
\\
=& a_{\vareps,k}^{\delta} +  \int_0^T\int_{\oeps} q_{\vareps} \psi   \widetilde{\left[T_k(\rhoeps)\right]_{\delta}} dx dt - \int_0^T \int_{\Omega} \psi \widetilde{p_0} \Myf\left(\overline{T_k(\rho)}\right) dx dt \\
& \quad + \int_0^T \int_{\oeps} (p(\rhoeps) - q_{\vareps}) \psi \widetilde{[T_k(\rhoeps)]_{\delta}} dx dt .
\end{aligned}
\end{align}
We estimate the last term on the right-hand side  of \eqref{ineq:aux_strong_conv_density_Aeps_k_11} using Proposition \ref{prop:improved_pressure_estimate_Bogovskii}. For this, we need the time-derivative of the test-function $\widetilde{[T_k(\rhoeps)]_{\delta}}$, and therefore the regularization in time was necessary. Our aim is to show that this term can be controlled by some $\vareps^{\varkappa}$ for some  $\varkappa>0$, uniform with respect to $\delta$. For this  we use similar arguments as in \cite[Section 7.9.5]{novotny2004introduction}. We notice that 
\begin{align}\label{eq:chain_rule_time}
    \partial_t \left( \psi \left[T_k(\rhoeps)\right]_{\delta} \right) = \psi' \left[T_k(\rhoeps) \right]_{\delta} + \psi \partial_t \left[T_k(\rhoeps)\right]_{\delta}.
\end{align}
From the renormalized continuity equation $\eqref{eq:renormalized_continuity_equation}$ we obtain  the regularized equation  ($b=T_k$ is admissible in $\eqref{eq:renormalized_continuity_equation}$)
\begin{align}\label{eq:continuity_equation_T_k_rhoeps}
\partial_t \left[T_k(\rhoeps)\right]_{\delta} + \div \left[ T_k(\rhoeps) \ueps \right]_{\delta} + \left[ \left(\rhoeps T_k'(\rhoeps) - T_k(\rhoeps) \right)\div \ueps \right]_{\delta} = 0 \qquad \mbox{ in } \mathcal{D}'(I' \times \R^3).
\end{align}
In other words, we have for all $\phi \in C_0^{\infty}(I')$ and $\eta \in C_0^{\infty}(\R^3)$ that 
\begin{align*}
\int_0^T \phi \int_{\R^3} \left[T_k(\rhoeps) \ueps \right]_{\delta} \cdot \nabla \eta dx dt = \int_0^T \int_{\R^3} \left[\partial_t \left[T_k(\rhoeps)\right]_{\delta} +  \left[\left(\rhoeps T_k'(\rhoeps) - T_k(\rhoeps) \right)\div \ueps \right]_{\delta}   \right] \eta dx dt.
\end{align*}
In particular, we have $\div \left[T_k(\rhoeps) \ueps \right]_{\delta} \in C^{\infty}(I',L^2(\R^3))$.  Using integration by parts we also obtain $\left[T_k(\rhoeps)\ueps \right]_{\delta} \cdot \nu = 0$ on $\partial \oeps$ (this means in $H^{-\frac12}(\partial \oeps)$) and everywhere in $I'$.

Now, using the estimate in Proposition \ref{prop:improved_pressure_estimate_Bogovskii} we get
\begin{align}
\begin{aligned}\label{ineq:aux_error_pressure_strong_conv_density}
\int_0^T \int_{\oeps} (p(\rhoeps) - q_{\vareps}) \psi \widetilde{[T_k(\rhoeps)]_{\delta}} dx dt \le& C \bigg\{\vareps^{\frac{\lambda}{2}} \|\Beps(\partial_t (\psi \widetilde{[T_k(\rhoeps)]_{\delta}})\|_{L^1((0,T),L^{\frac{2\gamma}{\gamma -1}}(\oeps))}   
\\
&+  \vareps^{\frac{\lambda}{s} - \frac{2\theta}{s} - 1} \vareps \|\nabla \Beps (\psi \widetilde{[T_k(\rhoeps)]_{\delta}})\|_{L^{s'}((0,T),L^{\frac{s'\gamma p}{(p-2)\gamma - p}}(\oeps))}
    \\
    &+ \vareps^{\frac{\lambda}{2}} \|\Beps(\psi \widetilde{[T_k(\rhoeps)]_{\delta}})\|_{W^{1,1}((0,T),L^{\frac{2\gamma}{\gamma -1}}(\oeps))} \bigg\}.
\end{aligned}
\end{align}
Let us consider the norms including the Bogovskii operator for $q = \frac{2\gamma }{\gamma -1} \in [2,6]$ to obtain with $\tilde \theta = \frac{3q - 6}{2q} = \frac{3}{2\gamma}$ from Corollary \ref{cor:Bogovskii} 
\begin{align*}
\big\|\Beps(\partial_t(\psi &\widetilde{[T_k(\rhoeps)]_{\delta}})) \big\|_{L^1((0,T),L^q(\oeps))} \\ 
\le & \|\Beps(\psi' \widetilde{[T_k(\rhoeps)]_{\delta}}) \|_{L^1((0,T),L^q(\oeps))} + \|\Beps(\div (\psi [T_k(\rhoeps) \ueps]_{\delta}))\|_{L^1((0,T),L^q(\oeps))}
\\
&+ \|\Beps( \psi \widetilde{[(\rhoeps T_k'(\rhoeps) - T_k(\rhoeps)) \div \ueps]_{\delta})} \|_{L^1((0,T),L^q(\oeps))}
\\
\le& C \|\psi'\|_{L^1(0,T)} \|\widetilde{[T_k(\rhoeps)]_{\delta}}\|_{L^{\infty}(I',L^q(\oeps))}
+ C \|\psi\|_{L^{\infty}(0,T)}\| [T_k(\rhoeps) \ueps]_{\delta}\|_{L^1(I',L^q(\oeps))}
\\
&+ C \vareps^{-\frac{3}{2\gamma}} \|\psi\|_{L^{\infty}(0,T)} \Big\{ \| \widetilde{[(\rhoeps T_k'(\rhoeps) - T_k(\rhoeps)) \div \ueps]_{\delta})} \|_{L^1(I',L^{\frac65}(\oeps))}
\\
&+ \vareps \| \widetilde{[(\rhoeps T_k'(\rhoeps) - T_k(\rhoeps)) \div \ueps]_{\delta})}\|_{L^1(I',L^2(\oeps))}
\Big\}.
\\
\le&  C \|\psi'\|_{L^1(0,T)} \|\widetilde{[T_k(\rhoeps)]_{\delta}}\|_{L^{\infty}(I',L^q(\oeps))}
+ C \|\psi\|_{L^{\infty}(0,T)}\| [T_k(\rhoeps) \ueps]_{\delta}\|_{L^1(I',L^q(\oeps))}
\\
&+ C \vareps^{-\frac{3}{2\gamma}} \|\psi\|_{L^{\infty}(0,T)}\| \widetilde{[(\rhoeps T_k'(\rhoeps) - T_k(\rhoeps)) \div \ueps]_{\delta})}\|_{L^1(I',L^2(\oeps))} .
\end{align*}
Using again Corollary \ref{cor:Bogovskii}, we get
\begin{align*}
\| \Beps(\psi\widetilde{[T_k(\rhoeps)]_{\delta}})\|_{L^1((0,T),L^q(\oeps))} \le C \|\psi\|_{L^{\infty}(0,T)} \| \widetilde{[T_k(\rhoeps)]_{\delta}}\|_{L^1(I',L^q(\oeps))}
\end{align*}
and 
\begin{align*}
\vareps \|\nabla  \Beps(\psi\widetilde{[T_k(\rhoeps)]_{\delta}})\|_{L^{s'}((0,T),L^{\frac{s'\gamma p}{(p-2)\gamma - p}}(\oeps))} \le C \|\psi\|_{L^{\infty}(0,T)} \| \widetilde{[T_k(\rhoeps)]_{\delta}}\|_{L^{s'}(I',L^{\frac{s'\gamma p}{(p-2)\gamma - p}}(\oeps))} .
\end{align*}
Now, using the above estimates in  $\eqref{ineq:aux_error_pressure_strong_conv_density}$ we have 
\begin{align*}
&\int_0^T  \int_{\oeps} (p(\rhoeps) - q_{\vareps}) \psi \widetilde{[T_k(\rhoeps)]_{\delta}} dx dt 
\\
&\le C \bigg( \vareps^{\frac{\lambda}{2}}\|\psi'\|_{L^1(0,T)} \|\widetilde{[T_k(\rhoeps) ]_{\delta}}\|_{L^{\infty}(I',L^{\frac{2\gamma}{\gamma -1}}(\oeps))}
+ \vareps^{\frac{\lambda}{2}} \|\psi\|_{L^{\infty}(0,T)}\| [T_k(\rhoeps) \ueps]_{\delta}\|_{L^1(I',L^{\frac{2\gamma}{\gamma -1}}(\oeps))}
\\
& \ + \vareps^{\frac{\lambda}{2} - \frac{3}{2\gamma}} \|\psi\|_{L^{\infty}(0,T)} \| \widetilde{[(\rhoeps T_k'(\rhoeps) - T_k(\rhoeps))\div \ueps]_{\delta})} \|_{L^1(I',L^2(\oeps))} \\
& \ + \vareps^{\frac{\lambda}{s} - \frac{2\theta}{s} - 1} \|\psi\|_{L^{s'}((0,T)} \| \widetilde{[T_k(\rhoeps)]_{\delta}}\|_{L^{s'}(I',L^{\frac{s'\gamma p}{(p-2)\gamma -p}}(\oeps))}
+ \vareps^{\frac{\lambda}{2}} \|\psi\|_{L^{\infty}(0,T)} \| \widetilde{[T_k(\rhoeps)]_{\delta}}\|_{L^1(I',L^{\frac{2\gamma}{\gamma -1}}(\oeps))} 
\bigg) .
\end{align*}
Now, we let $\delta \to 0$. Critical terms including time-derivatives of regularization do not occur in the estimates above. Using (see \cite[Lemma 6.5]{novotny2004introduction})
\begin{align*}
\left[ T_k(\rhoeps)\right]_{\delta} &\rightarrow T_k(\rhoeps) &\mbox{ in }& L^{\tilde{s}}(\overline{I}', L^q(\oeps)) \quad \mbox{ for } q,\tilde{s} \in [1,\infty),
\\
\left[ T_k(\rhoeps) \ueps\right]_{\delta} &\rightarrow T_k(\rhoeps) \ueps &\mbox{ in }& L^2(I',L^6(\oeps)),
\\
\left[ T_k(\rhoeps)\right]_{\delta} &\rightarrow T_k(\rhoeps) &\mbox{ in }& L^p(I',L^{\infty}(\oeps)) \quad \mbox{ for } p \in [1,\infty),
\\
\left[(\rhoeps T_k'(\rhoeps) - T_k(\rhoeps))\div \ueps \right]_{\delta} &\rightarrow (\rhoeps T_k'(\rhoeps) - T_k(\rhoeps))\div \ueps &\mbox{ in }& L^2(I',L^2(\oeps)) , 
\end{align*}
for any $k$ or $\vareps$, and same results are valid for the 
mean-values. Hence, for $\delta \to 0$ we obtain from $\eqref{ineq:aux_strong_conv_density_Aeps_k_11}$ that  (using $a_{\vareps,k}^{\delta} \to 0$ for $\delta \to 0$) 
\begin{align*}
A_{\vareps,k}^{1,1} &\le \int_0^T\int_{\oeps} q_{\vareps} \psi   \widetilde{T_k(\rhoeps)} dx dt - \int_0^T \int_{\Omega} \psi \widetilde{p_0} \Myf\left(\overline{T_k(\rho)}\right) dx dt 
\\
& \ + C \bigg( \vareps^{\frac{\lambda}{2}}\|\psi'\|_{L^1(0,T)} \|\widetilde{T_k(\rhoeps) }\|_{L^{\infty}((0,T),L^{\frac{2\gamma}{\gamma -1}}(\oeps))} 
+ \vareps^{\frac{\lambda}{2}} \|\psi\|_{L^{\infty}(0,T)}\| T_k(\rhoeps) \ueps\|_{L^1((0,T),L^{\frac{2\gamma}{\gamma -1}}(\oeps))}
\\
& \ + \vareps^{\frac{\lambda}{2} - \frac{3}{2\gamma}} \|\psi\|_{L^{\infty}(0,T)} \| \widetilde{(\rhoeps T_k'(\rhoeps) - T_k(\rhoeps))\div \ueps )} \|_{L^1((0,T),L^2(\oeps))} \\ 
& \ + \vareps^{\frac{\lambda}{s} - \frac{2\theta}{s} - 1} \|\psi\|_{L^\infty((0,T)} \| \widetilde{T_k(\rhoeps)}\|_{L^{s'}((0,T),L^{\frac{s'\gamma p}{(p-2)\gamma -p}}(\oeps))}
+ \vareps^{\frac{\lambda}{2}} \|\psi\|_{L^{\infty}(0,T)} \| \widetilde{T_k(\rhoeps)}\|_{L^1((0,T),L^{\frac{2\gamma}{\gamma -1}}(\oeps))}
\bigg) .
\end{align*}
Next, we  replace the function $\psi \in C_0^{\infty}(0,T)$ by $\psi = 1$ which is even though is not an admissible test-function for $\eqref{eq:renormalized_continuity_equation}$, since it has no compact support. However, the estimates above allow a density argument, more precisely we choose a sequence $\psi = \psi_m \in C_0^{\infty}(0,T)$ with $0 \le \psi_m \le 1$ and $\psi_m \rightarrow 1$ in $L^p(0,T)$ for every $p \in [1,\infty)$ and $\psi_m'$ is bounded in $L^1(0,T)$ (this is the only norm for which the derivative $\psi_m'$ occurs in the estimates above). Hence, in all calculations above we can replace $\psi$ by $\psi = 1$ (also on the left-hand side in $\eqref{ineq:aux_estimate_osc_defect}$), and in particular, we have
\begin{equation}\label{interm-A-11-k}
\begin{aligned}
A_{\vareps,k}^{1,1} \le& \int_0^T\int_{\oeps} q_{\vareps}   \widetilde{T_k(\rhoeps)} dx dt - \int_0^T \int_{\Omega}  \widetilde{p_0} \Myf\left(\overline{T_k(\rho)}\right) dx dt 
\\
&+ C \bigg( \vareps^{\frac{\lambda}{2}} \|\widetilde{T_k(\rhoeps) }\|_{L^{\infty}((0,T),L^{\frac{2\gamma}{\gamma -1}}(\oeps))}
+ \vareps^{\frac{\lambda}{2}} \| T_k(\rhoeps) \ueps\|_{L^1((0,T),L^{\frac{2\gamma}{\gamma -1}}(\oeps))}
\\
&+ \vareps^{\frac{\lambda}{2} - \frac{3}{2\gamma}}  \| \widetilde{(\rhoeps T_k'(\rhoeps) - T_k(\rhoeps))\div \ueps )} \|_{L^1((0,T),L^2(\oeps))} \\
&+ \vareps^{\frac{\lambda}{s} - \frac{2\theta}{s} - 1}  \| \widetilde{T_k(\rhoeps)}\|_{L^{s'}((0,T),L^{\frac{s'\gamma p}{(p-2)\gamma -p}}(\oeps))}
+ \vareps^{\frac{\lambda}{2}}  \| \widetilde{T_k(\rhoeps)}\|_{L^1((0,T),L^{\frac{2\gamma}{\gamma  -1}}(\oeps))} 
\bigg) . 
\end{aligned}
\end{equation} 
Now, we estimate all the norms in the right-hand side of \eqref{interm-A-11-k} including the cut-off function $T_k$. We have for all $s\geq 0$ 
\begin{align*}
    T_k(s) \le k, \qquad |sT_k'(s) - T_k(s)| \le 2k.
\end{align*}
This implies for a constant $C_k>0$ (which might explode for $k\to \infty$) and using the inequality $\eqref{ineq:Gagliardo_Nirenberg_ueps}$
\begin{align}\label{inter-A-11-k-2}
A_{\vareps,k}^{1,1} \le & \int_0^T\int_{\oeps} q_{\vareps}   \widetilde{T_k(\rhoeps)} dx dt - \int_0^T \int_{\Omega}  \widetilde{p_0} \Myf\left(\overline{T_k(\rho)}\right) dx dt 
\notag \\
&+ C_k \left(  \vareps^{\frac{\lambda}{2}} + \vareps^{\frac{\lambda}{2}} \| \ueps \|_{L^1((0,T),L^{\frac{2\gamma}{\gamma -1}}(\oeps))} + \vareps^{\frac{\lambda}{2} - \frac{3}{2\gamma}}\|\div \ueps \|_{L^1((0,T),L^2(\oeps))} + \vareps^{\frac{\lambda}{s} - \frac{2\theta}{s} - 1} \right)
\notag \\
\le& \int_0^T\int_{\oeps} q_{\vareps}   \widetilde{T_k(\rhoeps)} dx dt - \int_0^T \int_{\Omega}  \widetilde{p_0} \Myf\left(\overline{T_k(\rho)}\right) dx dt 
\notag \\
&+ C_k \left(  \vareps^{\frac{\lambda}{2}} + \vareps^{\frac{\lambda}{2} - \frac{3}{2\gamma}} + \vareps^{\frac{\lambda}{2} - \frac{3}{2\gamma} -1}+ \vareps^{\frac{\lambda}{s} - \frac{2\theta}{s} - 1} \right),
\end{align}
where  in the last inequality we used the a priori estimates for $\ueps$ from Proposition \ref{prop:apriori_basic} and $\eqref{ineq:Gagliardo_Nirenberg_ueps}$. Note that for any $k$, the terms in the second line of the above inequality tends to zero for $\vareps \to 0$ (under the assumption that $\lambda>1+ \frac{3}{\gamma}$, see also Remark \ref{rem:improved_pressure_estimate}).

It remains to estimate the term 
\begin{align*}
B_{\vareps,k}:= \int_0^T\int_{\oeps} q_{\vareps}   \widetilde{T_k(\rhoeps)} dx dt - \int_0^T \int_{\Omega}  \widetilde{p_0} \Myf\left(\overline{T_k(\rho)}\right) dx dt 
\end{align*}
in the right-hand side of \eqref{inter-A-11-k-2}. 
We emphasize that $q_{\vareps}$ converges to $p_0$ in the weak two-scale sense (see \eqref{weak-two-scale-q_eps}). Hence, if we would have the strong (two-scale) convergence of $T_k(\rhoeps)$ to $\overline{T_k(\rho)}$, we would obtain  that $B_{\vareps,k} \rightarrow 0$ for $\vareps \to 0$ for each $k$. To make this argument rigorous we proceed in the following way. We have
\begin{align*}
B_{\vareps,k} =& \int_0^T \int_{\oeps} q_{\vareps} \left[ \widetilde{T_k(\rhoeps)} - \frac{1}{|Y_f|} \widetilde{ \Myf(\overline{T_k(\rho)})}  \right] dx dt 
\\
+& \int_0^T \int_{\Omega} \left[ \frac{\chi_{\oeps} q_{\vareps}}{|Y_f|} - p_0 \right] \widetilde{ \Myf(\overline{T_k(\rho)})  } dx dt =: B_{\vareps,k}^1 + B_{\vareps,k}^2.
\end{align*}
Due to the two-scale convergence of $ \frac{\chi_{\oeps} q_{\vareps}}{|Y_f|} - p_0$ to $\left( \frac{\chi_{Y_f}}{|Y_f|} - 1\right) p_0$,  see again $\eqref{weak-two-scale-q_eps}$, we obtain (for each $k$) that $B_{\vareps,k}^2 \rightarrow 0 $ as $\vareps \to 0$. For the first term $B_{\vareps,k}^1$ we first notice that, due to Proposition \ref{prop:improved_pressure_estimate_Bogovskii}, we have that $q_{\vareps}$ is bounded in
\begin{align}\label{decomposition_qeps}
 L^2((0,T),H^1(\oeps)) + L^{\infty}((0,T),W^{1,\gamma}(\oeps)) + \vareps\left( L^2((0,T)\times \oeps) + L^{\infty}((0,T),L^{\gamma}(\oeps)) \right).
\end{align}
Without loss of generality we assume that $\gamma \in \left(\frac95,2\right]$ and we assume   (see also Remark \ref{rem:qeps_L2H1} below for the treatment of the additional term in the decomposition bounded in $L^2((0,T),H^1(\oeps))$)
\begin{align}\label{ineq:bound_qeps}
    \|q_{\vareps}\|_{L^{\infty}((0,T),W^{1,\gamma}(\oeps))} \le C.
\end{align}
Here, we emphasize that the parts in the decomposition of $q_{\vareps}$ bounded by $\vareps$ have no contribution for $\vareps \to 0$. Now, observe that  (we can suppress the mean value $\, \widetilde{\cdot}$)
\begin{align*}
B_{\vareps,k}^1 \le C \left\| T_k(\rhoeps) - \frac{1}{|Y_f|}  \Myf(\overline{T_k(\rho)})\right\|_{L^1((0,T),W^{1,\gamma}(\oeps)')}.
\end{align*}
We have
\begin{align*}
\bigg\| T_k(\rhoeps) - \frac{1}{|Y_f|}& \Myf (\overline{T_k(\rho)})\bigg\|_{L^1((0,T),W^{1,\gamma}(\oeps)')} \\
\le & \| T_k(\rhoeps) - \rhoeps\|_{L^1((0,T),W^{1,\gamma}(\oeps)')} + \left\|\rhoeps - \frac{1}{|Y_f| } \Myf (\rho_0) \right\|_{L^1((0,T),W^{1,\gamma}(\oeps)')}
\\
&+ \frac{1}{|Y_f|} \left\| \Myf (\rho_0) - \Myf (\overline{T_k(\rho)}) \right\|_{L^1((0,T),W^{1,\gamma}(\oeps)')}
\\
=&: \sum_{i=1}^3 B_{\vareps,k}^{1,i}.
\end{align*}
Due to the embedding $W^{1,\gamma}(\oeps) \subset L^{\gamma^{\ast}}(\oeps) $ with $\frac{1}{\gamma^{\ast}} = \frac{1}{\gamma} - \frac13$, we get $L^{(\gamma^{\ast})'}(\oeps) \cong L^{\gamma^{\ast}}(\oeps)' \subset W^{1,\gamma}(\oeps)'$ and therefore with $(\gamma^{\ast})' = \frac{3\gamma}{4\gamma - 3} $ 
\begin{align*}
\sup_{\vareps >0 } B_{\vareps,k}^{1,1} &\le \sup_{\vareps > 0}\|T_k(\rhoeps) - \rhoeps\|_{L^1((0,T),L^\frac{3\gamma}{4\gamma - 3}(\oeps))} \le h(k)
\end{align*}
with $h(k) \to 0$ for $k\to \infty$ (here $h(k)$ will be a generic function, independent of $\vareps$, which may change in the following, but always tending to $0$ for $k\to \infty$). The last inequality follows from  Proposition \ref{prop:apriori_pressure_refined}, Chebyshev and H\"older inequality, and the fact that $\frac{3\gamma}{4\gamma - 3} < \frac53 \gamma -1 = \gamma + \nu$. We refer to \cite[Exercise 7.59]{novotny2004introduction} for more details on the precise structure of $h(k)$.

\vspace{.2cm}

Next, we consider the term $B_{\vareps,k}^{1,3}$. With the same notation as above and similar arguments, we get
\begin{align*}
B_{\vareps,k}^{1,3} \le C \|\Myf (\rho_0) - \Myf(\overline{T_k(\rho)}) \|_{L^1((0,T),L^{(\gamma^{\ast})'}(\Omega))} 
\le C \|\rho_0 - \overline{T_k(\rho)} \|_{L^1((0,T),L^{\frac{3\gamma}{4\gamma - 3}}(\Omega \times Y_f))}.
\end{align*}
Due to Proposition \ref{prop:weak_compactness}, we have that $\chi_{\oeps} (\rhoeps - T_k(\rhoeps)) $ converges weakly in the two-scale sense to $\chi_{Y_f}(\rho - \overline{T_k(\rho)})$ in $L^s((0,T),L^{\gamma}(\Omega))$ for arbitrary $s \in [1,\infty)$, and therefore in particular in $L^1((0,T),L^{\frac{3\gamma}{4\gamma - 3}}(\Omega))$, since $\frac{3\gamma}{4\gamma - 3} < \gamma$ for $\gamma > \frac32$. Hence, due to the lower semicontinuity of the norm with respect to the weak two-scale convergence, we have
\begin{align*}
\|\rho_0 - \overline{T_k(\rho)}\|_{L^1((0,T),L^{\frac{3\gamma}{4\gamma - 3}}(\Omega \times Y_f))} \le \liminf_{\vareps \to 0} \|\rhoeps - T_k(\rhoeps)\|_{L^1((0,T),L^{\frac{3\gamma}{4\gamma - 3}}(\oeps))} \le h(k).
\end{align*}
The last inequality was already obtained above for the term $B_{\vareps,k}^{1,1}$.

\vspace{.2cm}

It remains to estimate the term $B_{\vareps,k}^{1,2}$. We emphasize that $B_{\vareps,k}^{1,2}$ is not depending on $k$. We find
\begin{align*}
B_{\vareps,k}^{1,2} &\le \|\chi_{\oeps} \rhoeps - \Myf(\rho_0)\|_{L^1((0,T),W^{1,\gamma}(\Omega)')} + \left\| \Myf (\rho_0) \left( 1 - \frac{\chi_{\oeps}}{|Y_f|}\right) \right\|_{L^1((0,T),W^{1,\gamma}(\Omega)')} \\
& =: B_{\vareps,k}^{1,2,1} + B_{\vareps,k}^{1,2,2}.
\end{align*}
It is well-known that $\left( 1 - \frac{\chi_{\oeps}}{|Y_f|}\right) \rightharpoonup 0$ weakly in $L^p(\Omega)$ for every $p\in [1,\infty)$ (even weakly$^{\ast}$ in $L^{\infty}(\Omega)$). Since the embedding $W^{1,p}(\Omega)$ into $L^p(\Omega)$ is compact, we obtain from the Schauder theorem, see \cite[X, 4]{Yosida1980}, that the embedding $ L^{p'}(\Omega) \hookrightarrow W^{1,p}(\Omega)'$ is compact. Hence, we obtain $\left( 1 - \frac{\chi_{\oeps}}{|Y_f|}\right) \rightarrow 0$ strongly in $W^{1,p}(\Omega)'$ for every $p \in (1,\infty)$. In particular, since this sequence is independent of time, this convergence is also valid in $L^1((0,T),W^{1,\gamma}(\Omega)')$. Since $\Myf (\rho_0) \in L^1((0,T),L^{\frac{3\gamma}{4\gamma - 3}}(\Omega))$, we obtain with Lemma \ref{lem:strong_convergence_product} that
\begin{align*}
B^{1,2,2}_{\vareps,k}=\left\| \Myf (\rho_0) \left( 1 - \frac{\chi_{\oeps}}{|Y_f|}\right) \right\|_{L^1((0,T),W^{1,\gamma}(\Omega)')} \rightarrow 0 \, \text{ as } \,  \vareps \to 0. 
\end{align*}
Next, we  estimate the term $B_{\vareps,k}^{1,2,1}$. Here we follow ideas from \cite{feireisl2010homogenization}, but provide additional details.  First of all, we notice that from $\eqref{eq:weak_mass_balance}$ we obtain for every $\psi \in C_0^{\infty}([0,T)\times \overline{\Omega})$ that for any $\tau \in [0,T]$ it holds that (remember that we assumed $\rhoeps \in C([0,T],L_{\mathrm{weak}}^{\gamma}(\oeps))$ for a weak solution)
\begin{align*}
\int_0^{\tau} \int_{\Omega} (\chi_{\oeps}\rhoeps \partial_t \psi + \chi_{\oeps} \rhoeps \ueps \cdot \nabla \psi) dx dt = \int_{\Omega} \chi_{\oeps} \rhoeps(\tau) \psi(\tau) dx - \int_{\Omega} \chi_{\oeps} \rho_{\vareps,0} \psi(0) dx.
\end{align*}
In particular, choosing $\psi(t,x) := \eta(x) \xi(t)$ with $\xi \in C_0^{\infty}((0,T))$ and $\eta \in C_0^{\infty}(\overline{\Omega})$ we obtain by integration by parts that
\begin{align*}
\frac{d}{dt} \int_{\Omega}  \chi_{\oeps} \rhoeps \eta dx = \int_{\Omega} \chi_{\oeps} \rhoeps \ueps \cdot \nabla \eta dx \qquad \mbox{in } \mathcal{D}'((0,T))
\end{align*}
and for all $\eta \in C_0^{\infty}(\overline{\Omega})$. Due to the a priori estimates in Proposition \ref{prop:apriori_basic} we obtain, since $\gamma > \frac95$ and therefore $\gamma + \nu = \frac53 \gamma - 1>2$,
\begin{align*}
\int_0^T \left| \int_{\Omega} \chi_{\oeps} \rhoeps \ueps \cdot \nabla \eta dx \right| dt  &\le C \|\rhoeps\|_{L^{\gamma + \nu}((0,T)\times \oeps)} \|\ueps\|_{L^2((0,T)\times \oeps)} \|\nabla \eta \|_{L^{\frac{10\gamma - 6}{5\gamma -9}}(\Omega)}
\\
&\le C \|\nabla \eta \|_{L^{\frac{10\gamma - 6}{5\gamma -9}}(\Omega)}.
\end{align*}
This implies $\chi_{\oeps} \rhoeps \in W^{1,1}((0,T),W^{1,\frac{10\gamma - 6}{5\gamma -9}}(\Omega)') \hookrightarrow C^0([0,T],W^{1,\frac{10\gamma - 6}{5\gamma -9}}(\Omega)')$ bounded. Further, $\chi_{\oeps}\rhoeps $ is bounded in $L^{\infty}((0,T),L^{\gamma}(\Omega))$. Arguing as in the proof of \cite[Lemma 6.2]{novotny2004introduction}, we obtain up to a subsequence for a $\tilde{\rho} \in C^0([0,T],L_{\mathrm{weak}}^{\gamma}(\Omega))$
\begin{align*}
\chi_{\oeps} \rhoeps \rightarrow \tilde{\rho} \qquad\mbox{in } C^0([0,T],L_{\mathrm{weak}}^{\gamma}(\Omega)).
\end{align*}
The weak convergence of $\chi_{\oeps} \rhoeps $ to $\Myf(\rho_0)$ in $L^s((0,T),L^{\gamma}(\Omega)$ for arbitrary $s\in [1,\infty)$, which is a consequence of the weak two-scale convergence of $\chi_{\oeps}\rhoeps$, implies $\tilde{\rho} = \Myf(\rho_0)$ and therefore $\chi_{\oeps} \rhoeps \rightarrow  \Myf(\rho_0)$ in $C^0([0,T],L_{\mathrm{weak}}^{\gamma}(\Omega))$. Using the compactness of the embedding $W^{1,\gamma}(\Omega)$ into $L^{\gamma'}(\Omega)$ for $\gamma > \frac95$, we obtain that $L^{\gamma}(\Omega) \hookrightarrow W^{1,\gamma}(\Omega)'$ compact (using again the Schauder theorem). Hence, with the same arguments as in the proof of \cite[Lemma 6.4]{novotny2004introduction}, we obtain for every $s \in [1,\infty)$
\begin{align*}
\chi_{\oeps} \rhoeps \rightarrow  \Myf(\rho_0) \qquad\mbox{strongly in } L^s((0,T),W^{1,\gamma}(\Omega)') .
\end{align*}
Hence, we obtain 
\begin{align*}
B_{\vareps,k}^{1,2,1} = \|\chi_{\oeps}\rhoeps - \Myf(\rho_0)\|_{L^1((0,T),W^{1,\gamma}(\Omega)')} \rightarrow 0
\end{align*}
for $\vareps \to 0$. In summary, we obtain
\begin{align*}
    B_{\vareps,k} \le \Delta_{\vareps,k} + h(k),
\end{align*}
where $\Delta_{\vareps,k} \rightarrow 0$ as $\vareps \to 0$ for each  $k$. Using $\eqref{ineq:aux_estimate_osc_defect}$, we finally end up with
\begin{align}\label{ineq:aux_estimate_osc_defect2}
     \left\| T_k \left(\teps \rhoeps\right) - T_k(\rho_0)\right\|_{L^{\gamma + 1}((0,T)\times \Omega \times Y_f)} \le \Delta_{\vareps,k} + h(k).
\end{align}
\begin{remark}\label{rem:qeps_L2H1}
We made without loss of generality the assumption $\eqref{ineq:bound_qeps}$. However, in the decomposition $\eqref{decomposition_qeps}$ also the bound in $L^2((0,T),H^1(\oeps))$ is critical. Let us shortly sketch that the previous arguments also hold in this case (the critical part is now the time-integrability, since $H^1(\oeps) \hookrightarrow W^{1,\gamma}(\oeps)$ for $\gamma \le 2$). We have to bound the term
\begin{align*}
\left\| T_k(\rhoeps) - \frac{1}{|Y_f|} \Myf (\overline{T_k(\rho)})\right\|_{L^2((0,T),H^1(\oeps)')}.
\end{align*}
Due to the embedding $L^{\frac65}(\oeps) \hookrightarrow H^1(\oeps)'$, we can consider in the terms $B_{\vareps,k}^1$ and $B_{\vareps,k}^3$ the norms on $L^2((0,T),L^{\frac65}(\oeps))$, and since $\rhoeps $ is  uniformly bounded in $L^{\infty}((0,T),L^{\gamma}(\oeps)) \hookrightarrow L^2((0,T),L^{\frac65}(\oeps))$, all the calculations above remain valid.
\end{remark}

To conclude, we use
\begin{align*}
\|\teps \rhoeps - \rho_0 \|_{L^1((0,T)\times \Omega \times Y_f)} \le& \|\teps \rhoeps - T_k(\teps \rhoeps) \|_{L^1((0,T)\times \Omega \times Y_f)} 
+ \| T_k(\teps \rhoeps) - T_k(\rho_0)\|_{L^1((0,T)\times \Omega \times Y_f)} 
\\
&+ \|T_k(\rho_0) - \rho_0\|_{L^1((0,T)\times \Omega \times Y_f)} =: \sum_{i=1}^3 C_{\vareps,k}^i.
\end{align*}
Using $T_k(\teps \rhoeps) = \teps T_k(\rhoeps)$, the properties of the unfolding operator, and similar arguments as above for the terms $B_{\vareps,k}^1$ and $B_{\vareps,k}^3$, we obtain $\sup_{\vareps >0} \left( C_{\vareps,k}^1 +C_{\vareps,k}^3\right) \le h(k)$. Further, using $\eqref{ineq:aux_estimate_osc_defect2}$ and the H\"older inequality, we get
\begin{align*}
C_{\vareps,k}^2 \le \Delta_{\vareps,k} +  h(k) 
\end{align*}
with $\Delta_{\vareps,k} \to 0$  as $\vareps \to 0$ for each $k$. Hence, we obtain $\teps \rhoeps \rightarrow \rho_0$ in $L^1((0,T)\times \Omega \times Y_f)$. Using interpolation and the boundedness of $\rhoeps $ in $L^{\gamma + \nu}((0,T)\times \oeps)$ from Proposition \ref{prop:apriori_pressure_refined}, we obtain that $\teps \rhoeps \rightarrow \rho_0 $ in $L^s((0,T)\times \Omega \times Y_f)$ for every $s \in [1,\gamma + \nu)$.
\end{proof}

\section{Derivation of the macroscopic model}
\label{sec:derivation_macro_model}

Based on the compactness results in Section \ref{sec:compactness} we now derive the limit model for $\vareps \to 0$. 

\begin{proposition}\label{prop:div_y_rho_0_u_0}
For the limit functions $\rho_0$ and $\uvect_0$ from Proposition \ref{prop:weak_compactness} it holds almost everywhere in $(0,T)\times \Omega \times Y_f$ that
\begin{align*}
\div_y(\rho_0 \uvect_0) = 0.
\end{align*}
\end{proposition}
\begin{proof}
We test the mass balance equation $\eqref{eq:weak_mass_balance}$ with $\psieps(t,x):= \vareps \psi\left(t,x,\fxe\right)$ with $\psi \in C_0^{\infty}((0,T)\times \Omega \times Y_f)$ and obtain
\begin{align*}
\int_0^T \int_{\oeps} \vareps \rhoeps \partial_t \psi\left(t,x,\fxe\right) dx dt + \vareps \int_0^T \int_{\oeps} \rhoeps \ueps \cdot \left[ \nabla_x \psi + \foe \nabla_y \psi\right] \left(t,x,\fxe\right) dx dt = 0.
\end{align*}
For $\vareps \to 0$ we obtain with the compactness results from Propositions \ref{prop:weak_compactness} and \ref{prop:strong_compactness_rhoeps} that
\begin{align*}
\int_0^T \int_{\Omega} \int_{Y_f} \rho_0 \uvect_0 \cdot \nabla_y \psi dy dx dt = 0.
\end{align*}
This implies the desired result.
\end{proof}

Next, we test the impulse equation $\eqref{eq:weak_impulse}$ with $\phieps(t,x):= \phi\left(t,x,\fxe\right)$ with $\phi \in C_0^{\infty}([0,T) \times \Omega, C_{\per}^{\infty}(Y))^3$ with $\phi = 0$ in $Y_s$ and $\nabla_y \cdot \phi = 0$ to obtain
\begin{align*}
\int_0^T &\int_{\oeps} \bigg\{ \vareps^{\lambda}\rhoeps  \ueps \cdot \partial_t \phi \left(t,x,\fxe\right) + \vareps^{\lambda}\rhoeps \ueps \otimes \ueps : \left[ \nabla_x \phi\left(t,x,\fxe\right) + \foe \nabla_y \phi\left(t,x,\fxe\right)\right]
\\
-& \vareps^2 \mathbb{S}(\nabla \ueps ) : \left[ \nabla_x \phi\left(t,x,\fxe\right) + \foe \nabla_y \phi\left(t,x,\fxe\right)\right]
\\
+& p(\rhoeps) \div_x \phi\left(t,x,\fxe\right)  + \rhoeps \nabla F \cdot \phi \left(t,x,\fxe\right) dx dt
\\
&= - \int_{\oeps} \vareps^{\lambda} \m_{\vareps,0} \phi\left(0,x,\fxe\right)dx.
\end{align*}
We pass to the limit $\vareps \to 0$ separately in each term. 

For the first term we use  the strong two-scale convergence of $\rhoeps$ from Proposition \ref{prop:strong_compactness_rhoeps} and the weak two-scale convergence of $\ueps$ from Proposition \ref{prop:weak_compactness}. 
First of all, we have for $\gamma \geq 2$:
\begin{align*}
\left|\int_0^T \int_{\oeps} \vareps^{\lambda} \rhoeps \ueps\cdot \partial_t \phi \left(t,x,\fxe\right)dx dt \right| \le C \vareps^{\lambda} \|\rhoeps\|_{L^{\infty}((0,T),L^{\gamma}(\oeps))} \|\ueps\|_{L^2((0,T)\times \oeps)}  \le C\vareps^{\lambda}.
\end{align*}
Secondly, for $\gamma < 2$ we use Gagliardo-Nirenberg inequality with $\theta=\frac{6-3\gamma}{2\gamma}$ to obtain 
\begin{align*}
\|\ueps\|_{L^{\frac{\gamma}{\gamma -1}}(\oeps)} \le C\|\nabla \ueps\|^{\frac{6-3\gamma}{2\gamma}}_{L^2(\oeps)} \|\ueps\|^{\frac{5\gamma - 6}{2\gamma}}_{L^2(\oeps)} + \|\ueps\|_{L^2(\oeps)}
\end{align*}
so that 
\begin{align*}
\left|\int_0^T \int_{\oeps} \vareps^{\lambda} \rhoeps \ueps\cdot \partial_t \phi \left(t,x,\fxe\right)dx dt \right|  &\le C \vareps^{\lambda} \|\rhoeps\|_{L^{\infty}((0,T),L^{\gamma}(\oeps))} \|\ueps\|_{L^2((0,T),L^{\frac{\gamma}{\gamma-1}}(\oeps))}
\\
&\le C \vareps^{\lambda} \left( \|\nabla \ueps\|^{\frac{6-3\gamma}{2\gamma}}_{L^2((0,T)\times \oeps)}\|\ueps\|^{\frac{5\gamma-6}{2\gamma}}_{L^2((0,T)\times \oeps)} + \|\ueps\|_{L^2((0,T)\times \oeps)} \right) \\
&\le C\vareps^{\lambda} \left( \vareps^{\frac{3\gamma - 6}{2\gamma}} +1\right).
\end{align*}
Hence, this term vanishes for $\vareps \to 0$, as we consider $\lambda > 1+\frac{3}{\gamma}$.

For the nonlinear term we use similar arguments as above (the term including $\nabla_x \phi$ is one order higher in $\vareps$ and we ignore it) and consequently
\begin{align*} 
\bigg|\int_0^T &\int_{\oeps} \vareps^{\lambda - 1} \rhoeps \ueps \otimes \ueps : \nabla_y \phi\left(t,x,\fxe\right) dxdt \bigg|
\\
&\le C\vareps^{\lambda - 1} \|\rhoeps\|_{L^{\infty}((0,T),L^{\gamma}(\oeps))} \|\ueps\|^2_{L^2((0,T),L^{\frac{2\gamma}{\gamma - 1}}(\oeps))} 
\\
&\le C \vareps^{\lambda-1} \left(  \|\nabla \ueps\|^{\frac{3}{\gamma}}_{L^2((0,T)\times \Omega_\vareps)} \|\ueps\|^{\frac{2\gamma-3}{\gamma}}_{L^2((0,T)\times \Omega_\vareps)} + \|\ueps\|^2_{L^2((0,T)\times \Omega_\vareps)}\right) \\ 
&\le C\vareps^{\lambda - 1} \left( \vareps^{-\frac{3}{\gamma}} + 1\right) \le C \vareps^{\lambda -1- \frac{3}{\gamma}} \overset{\vareps \to 0}{\longrightarrow} 0 \ \ \textnormal{ since }\lambda> 1+\frac{3}{\gamma},
\end{align*}
where we have used the Gagliardo-Nirenberg inequality (with $\theta=\frac{3}{2\gamma}$): 
\begin{align*}
\|\ueps\|_{L^{\frac{2\gamma}{\gamma-1}}(\Omega_\vareps)} \leq C \|\nabla \ueps\|^{\frac{3}{2\gamma}}_{L^2(\oeps)} \|\ueps\|^{\frac{2\gamma-3}{2\gamma}}_{L^2(\oeps)} + \|\ueps\|_{L^2(\oeps)}. 
\end{align*}

Next,
from the compactness result in Proposition \ref{prop:weak_compactness} we obtain with $\mathbb I \textnormal{d}:\nabla_y \phi= \div_y \phi = 0$
\begin{align*}
-\vareps \int_0^T \int_{\oeps}  \mathbb{S}(\nabla \ueps) : \nabla_y \phi\left(t,x,\fxe\right) dx dt &= -\mu \int_0^T \int_{\oeps} \vareps (\nabla \ueps + \nabla \ueps^T ) : \nabla_y \phi\left(t,x,\fxe\right) dx dt 
\\
&\overset{\vareps \to 0}{\longrightarrow} - 2\mu \int_0^T \int_{\Omega} \int_{Y_f} D_y(\uvect_0) : \nabla_y \phi \, dy dx dt
\end{align*}
with the symmetric gradient $D_y(\uvect_0) = \frac12 (\nabla_y \uvect_0 + \nabla_y \uvect_0^\top)$. For the pressure term we have using the strong convergence of the density from Proposition \ref{prop:strong_compactness_rhoeps}
\begin{align*}
\int_0^T \int_{\oeps} p(\rhoeps) \div_x \phi \left(t,x,\fxe\right) dxdt  \overset{\vareps\to 0}{\longrightarrow}& \int_0^T \int_{\Omega} \int_{Y_f} p_0 \div_x \phi \, dy dx dt 
\\
=& -\int_0^T \int_{\Omega} \int_{Y_f} \nabla p_0 \cdot \phi dy dx dt.
\end{align*}
Finally, for the terms including the data we have 
\begin{align*}
-\int_0^T &\int_{\oeps} \rhoeps \nabla F \cdot \phi\left(t,x,\fxe\right) dx dt - \int_{\oeps} \vareps^{\lambda} \m_{\vareps,0} \phi \left(0,x,\fxe\right) dx 
\\
&\overset{\vareps\to 0}{\longrightarrow} -\int_0^T\int_{\Omega} \int_{Y_f} \rho_0 \nabla F \cdot \phi dy dx dt. 
\end{align*}
Altogether, we obtain
\begin{align*}
2\mu \int_0^T \int_{\Omega} \int_{Y_f} D_y(\uvect_0) : D_y( \phi) dy dx dt 
= \int_0^T\int_{\Omega} \int_{Y_f}  \left[\rho_0 \nabla F - \nabla p_0\right] \cdot \phi dy dx dt 
\end{align*}
and by density this result is valid for all $\phi \in L^2((0,T)\times \Omega,H_{\per}^1(Y_f))^3 \cap L^2((0,T),L^{\gamma'}(\Omega,H^1_{\per}(Y_f)))^3$ with $\phi = 0$ on $\Gamma$ and $\div_y \phi = 0$ (remember that we have $\nabla p_0 \in L^2((0,T)\times \Omega)^3 + L^{\infty}((0,T),L^{\gamma}(\Omega))^3$).

Hence, there exists $p_1 \in L^2((0,T),L^q(\Omega,L^2(Y_f)))$ with $q = \min\{2,\gamma\}$, such that  
\begin{align*}
2\mu \int_0^T \int_{\Omega} \int_{Y_f} D_y(\uvect_0) : D_y( \phi) dy dx dt  - \int_0^T \int_{\Omega} \int_{Y_f} p_1 \operatorname{div}_y  \phi \, dy dx dt
= \int_0^T\int_{\Omega} \int_{Y_f}  \left[\rho_0 \nabla F - \nabla p_0\right] \cdot \phi dy dx dt 
\end{align*}
for all $\phi \in L^2((0,T)\times \Omega,H_{\per}^1(Y_f))^3 \cap L^2((0,T),L^{\gamma'}(\Omega,H^1_{\per}(Y_f)))^3$ with $\phi = 0$ on $\Gamma$.
In particular, we obtain for almost every $(t,x) \in (0,T)\times \Omega$
\begin{align*}
2\mu \int_{Y_f} D_y(\uvect_0) : D_y( \phi) dy  - \int_{Y_f} p_1 \operatorname{div}_y  \phi \, dy
=  \int_{Y_f}  \left[\rho_0 \nabla F - \nabla p_0\right] \cdot \phi dy 
\end{align*}
for all $\phi \in H_{\per}^1(Y_f)^3$ with $\phi = 0 $ on $\Gamma$. Since $\operatorname{div}_y (\rho_0 \uvect_0) = 0$, we obtain $(\uvect_0,p_1)$ is a solution of the problem
\begin{align*} 
    -2\mu \div_y (D_y( \uvect_0)) + \nabla_y p_1 &= \rho_0 \nabla F - \nabla p_0 &\mbox{ in }& \{\rho_0 > 0\}  \times Y_f,
\\
\div_y( \uvect_0) &= 0 &\mbox{ in }&  \{\rho_0 > 0\}  \times Y_f,
\\
\uvect_0 &= 0 &\mbox{ on }& \{\rho_0 > 0\}  \times \Gamma,
\\
(\uvect_0,p_1) \mbox{ are } Y\mbox{-periodic}.
\end{align*}
We emphasize that this problem only has a unique weak solution pointwise almost everywhere in $(0,T)\times \{\rho_0 > 0\} $ and we have no information on the set $\{ \rho_0 = 0\}$. By the linearity of the previous equation, we obtain for almost every $(t,x,y) \in \{\rho_0 >0\} \times Y_f$
\begin{align}\label{eq:representation_u0}
\uvect_0(t,x,y) &= \sum_{i=1}^3 \left[\rho_0(t,x) \nabla F(t,x) - \nabla_x p_0(t,x) \right]_i q_i,
\\
\label{eq:representation_p1}
p_1(t,x,y) &= \sum_{i=1}^3 \left[\rho_0(t,x) \nabla F(t,x) - \nabla_x p_0(t,x) \right]_i \pi_i,
\end{align}
where $(q_i,\pi_i) \in H_{\per}^1(Y_f)^3 \times L_0^2(Y_f)$ is the unique weak solution of the cell problem
\begin{align*}
-2\mu \div_y(D_y(q_i)) + \nabla_y \pi_i &= e_i &\mbox{ in }& Y_f,
\\
\div_y q_i &= 0 &\mbox{ in }& Y_f,
\\
q_i &= 0 &\mbox{ on }& \Gamma,
\\
(q_i,\pi_i) \mbox{ are } Y\mbox{-periodic}.
\end{align*}
We define the Darcy velocity via averaging with respect to $Y_f$ to obtain
\begin{align*}
 \bar{\uvect}_0(t,x):=\int_{Y_f} \uvect_0(t,x,y) dy = K\left[\rho_0 \nabla F - \nabla_x p_0 \right](t,x)
\end{align*}
with the permeability tensor $K\in \R^{3\times 3}$ defined by (for $i,j=1,2,3$)
\begin{align}\label{def:permeability_tensor}
K_{ij}:= \int_{Y_f} D_y(q_i): D_y(q_j) dy.
\end{align}
Since $K$ is positive, we obtain 
\begin{align*}
    \nabla_x p_0 = \rho_0 \nabla F - K^{-1} \bar{\uvect}_0 \qquad \mbox{a.e. in } \{\rho_0 > 0\}.
\end{align*}
Since $\rho_0 \nabla F - K^{-1} \bar{\uvect}_0 \in L^2((0,T)\times \Omega)^3$ (remember that $\gamma > \frac95$ and $\rho_0 \in L^{\gamma + \nu}((0,T)\times \Omega)$ with $\nu = \frac23 \gamma - 1$), we obtain
\begin{align}
   \rho_0 \nabla p_0 \in L^s((0,T)\times \Omega)^3
\end{align}
for a suitable $s>1$.

Finally, we test the mass balance equation $\eqref{eq:weak_mass_balance}$ with $\psieps(t,x) = \psi(t,x)$ with $\psi \in C_0^{\infty}([0,T) \times \overline{\Omega})$ to obtain
\begin{align*}
\int_0^T \int_{\oeps} \rhoeps \partial_t \psi  + \rhoeps  \ueps \cdot \nabla_x \psi dx dt = - \int_{\oeps} \rho_{\vareps,0} \psi(0) dx.
\end{align*}
For $\vareps \to 0$ we get with Propositions \ref{prop:weak_compactness} and \ref{prop:strong_compactness_rhoeps} together with the assumptions on $\rho_{\vareps,0}$ 
\begin{align*}
\int_0^T \int_{\Omega} |Y_f| \rho_0 \partial_t \psi + \rho_0 \bar{\uvect}_0 \cdot \nabla_x \psi dx dt = - \int_{\Omega} |Y_f| \rho_{0,0} \psi(0) dx .
\end{align*}
By density, this result is valid for all $\phi \in L^{s'}((0,T), W^{1,s'}(\Omega))$ with $\partial_t \phi \in L^{(\gamma + \nu)'}((0,T)\times \Omega)$ and $\phi(T) = 0$. The corresponding strong formulation with the representation for $\rho_0 \bar{\uvect}_0$ is 
\begin{align*}
|Y_f|\partial_t \rho_0 + \div_x \left(K \rho_0^2 \nabla F - a\rho_0\nabla \rho_0^{\gamma}\right) = 0 \qquad \mbox{in } (0,T)\times \Omega
\end{align*}
with the initial condition $\rho_0(0) = |Y_f| \rho_{0,0}$.

\vspace{.5cm}
\noindent
{\bf Acknowledgments.} 
The work of K.B. and \v{S}.N. are supported by  the Praemium Academiae of \v{S}.N. The Institute
of Mathematics, CAS is supported by RVO:67985840.

\begin{appendix}
\section{Auxiliary results}\label{sec:appendix}

The appendix contains some  classical technical results which are important for our homogenization work.

\subsection{Poincar\'e inequality}

It holds the following well-known Poincar\'e inequality for functions with vanishing traces on $\geps$:
\begin{lemma}[Poincar\'e inequality]\label{lem:Poincare_inequality}
Let $p \in [1,\infty)$. For every $\phieps \in W^{1,p}(\oeps)$ with $\phieps  =0 $ on $\geps$ it holds that
\begin{align*}
\|\phieps \|_{L^p(\oeps)} \le C \vareps \| \nabla \phieps\|_{L^p(\oeps)}.
\end{align*}
\end{lemma}

\subsection{The Bogovskii operator}
\label{sec:Bogovskii_operator}

In this section we introduce a Bogovskii operator $\Beps : {L_0^p(\oeps)} \rightarrow W_0^{1,p}(\oeps)^3$ for $p \in (1,\infty)$ such that $\div \Beps \feps = \feps $ and 
\begin{align*}
\vareps \|\nabla \Beps \feps \|_{L^p(\oeps)} \le C \left( \|\feps\|_{W^{1,p'}(\oeps)'} + \vareps \|\feps\|_{L^p(\oeps)}\right)
\end{align*}
with a constant $C>0$ independent of $\vareps$. 
For the construction we use the restriction operator $\Reps: W_0^{1,p}(\Omega)^3 \rightarrow W_0^{1,p}(\oeps)^3$ from \cite{Lu-Schwarzacher} (below we give more details)  and the standard Bogovskii operator $\B: L^p_0(\Omega) \rightarrow W_0^{1,p}(\Omega)^3$,  see \cite{BogovskiSolutionFirstBVP} and \cite{novotny2004introduction}, which fulfills 
\begin{align*}
\div \B f = f,\qquad \|\nabla \B f\|_{L^p(\Omega)} \le C \|f\|_{L^p(\Omega)}
\end{align*}
for all $f \in L_0^p(\Omega)$. This operator can be extended to a bounded linear operator 
\begin{align*}
\B: \left\{ f \in W^{1,p'}(\Omega)' \, : \, \langle f,1\rangle_{W^{1,p'}(\Omega)',W^{1,p'}(\Omega)} = 0 \right\} \rightarrow L^p(\Omega),
\end{align*}
such that 
\begin{align*}
\|\B f\|_{L^p(\Omega)} \le C \|f\|_{W^{1,p}(\Omega)'}.
\end{align*}

For the restriction operator we use the operator constructed in the proof of \cite[Proposition 2.2]{Lu-Schwarzacher}, which is defined for perforated domains such that $\Omega \setminus \oeps$ is connected and fulfills additional important properties compared to standard restriction operators defined in the literature, see for example \cite{tartar1980incompressible} and \cite{allaire1989homogenization}.  First of all, for $u \in W_0^{1,q}(\Omega)^3$ with $q \in (1,\infty)$ we have 
\begin{align*}
    \Reps(u) \in W_0^{1,q}(\oeps)^3, \qquad \vareps \|\nabla \Reps(u)\|_{L^q(\oeps)} \le C \left( \|u\|_{L^q(\Omega_\vareps)} + \vareps \|\nabla u\|_{L^{q}(\oeps)}\right) .
\end{align*}
Further, for $\div u = 0$ in $\Omega \setminus \oeps$ it holds that
\begin{align*}
\div \Reps(u) = \div u.
\end{align*}
While the previous properties are quite standard, the following is of particular importance for the homogenization of compressible fluids: let $g \in L^r(\oeps)^3$ for $r > \frac32$ with $\div g \in L^q(\oeps)$ for $q \in (1,\infty)$ and $g \cdot \nu =0 $ on $\partial \oeps$. In particular, we can extend $g$ by zero to a function $\tilde{g}$ having the same regularity. Then, it holds that
\begin{align}\label{ineq:Restriction_operator_div}
 \|\Reps(\mathcal{B}(\div \tilde{g}))\|_{L^r(\oeps)} \le C \|g\|_{L^r(\oeps)}.
\end{align}
Now, we can define the operator $\Beps$ via the composition $\Beps(\feps):= \Reps \circ  B \tfeps$ for $\feps\in L_0^p(\oeps)$, where $\tfeps$ denotes the zero extension of $\feps$ to the whole domain $\Omega$. Using the properties of the restriction operator  we obtain immediately $\div \Beps \feps =  \feps$. Further, we have the inequality 
\begin{align*}
\vareps \|\nabla \Beps \feps \|_{L^p(\oeps)} &=  \vareps \|\nabla \Reps (\B \tfeps)\|_{L^p(\oeps)} \le C \left( \|\B \tfeps\|_{L^p(\Omega)} + \vareps \|\nabla \B \tfeps\|_{L^p(\Omega)}\right)
\\
&\le C \left( \|\tfeps \|_{W^{1,p'}(\Omega)'} + \vareps \|\tfeps\|_{L^p(\Omega)} \right)
\\
&\le C \left( \|\feps\|_{W^{1,p'}(\oeps)'} + \vareps \|\feps\|_{L^p(\oeps)} \right).
\end{align*}
In summary, we obtain the following result.
\begin{lemma}\label{lem:Bogovskii_Beps}
Let $p\in (1,\infty)$. There exists a linear and bounded operator  $\Beps : L_0^p(\oeps) \rightarrow W_0^{1,p}(\oeps)^3$ such that for all $\feps \in L_0^p(\oeps)$ it holds that
\begin{align*}
\div \Beps \feps = \feps
\end{align*}
and
\begin{align}\label{ineq:Bogovski_lemma}
\vareps \|\nabla \Beps \feps \|_{L^p(\oeps)} \le C \left( \|\feps\|_{W^{1,p'}(\oeps)'} + \vareps \|\feps\|_{L^p(\oeps)} \right)
\end{align}
with a constant $C>0$ independent of $\vareps$.
\end{lemma}

\begin{remark}\label{rem:Bogovskii}
Under the assumptions in Lemma \ref{lem:Bogovskii_Beps}, we can also consider the operator $\Beps$ as an operator from $L^p(\oeps)$ (without mean-value zero condition) to $W_0^{1,p}(\oeps)^3$ such that for every $\feps \in L^p(\oeps)$ it holds that
\begin{align*}
\div\Beps \feps = \feps - \frac{1}{|\oeps|}\int_{\oeps} \feps dx,
\end{align*}
and the estimate  $\eqref{ineq:Bogovski_lemma}$ is valid. More precisely, we define  
\begin{align*}
\tBeps \feps := \Beps \bigg( \underbrace{\feps - \frac{1}{|\oeps|}\int_{\oeps} \feps dx}_{:= \bar{f}_{\vareps}}\bigg).
\end{align*}
It is obvious that
\begin{align*}
\|\bar{f}_{\vareps}\|_{L^p(\oeps)} \le C \|\feps \|_{L^p(\oeps)}.
\end{align*}
Further, we have for every $\phieps \in W^{1,p'}(\oeps)$ that
\begin{align*}
\langle \bar{f}_{\vareps}, \phieps \rangle_{W^{1,p'}(\oeps)',W^{1,p'}(\oeps)} &= \int_{\oeps} \feps \phieps dx + \frac{1}{|\oeps|} \int_{\oeps} \feps dx \cdot \int_{\oeps} \phieps dx 
\\
&\le \|\feps\|_{W^{1,p'}(\oeps)'} \|\phieps\|_{W^{1,p'}(\oeps)}  + C\|\feps\|_{W^{1,p'}(\oeps)'} \|1\|_{W^{1,p'}(\oeps)} \|\phieps \|_{L^1(\oeps)}
\\
&\le C\|\feps\|_{W^{1,p'}(\oeps)'} \|\phieps\|_{W^{1,p'}(\oeps)}.
\end{align*}
Hence, we have 
\begin{align*}
\|\bar{f}_{\vareps}\|_{W^{1,p'}(\oeps)'} \le C \|\feps\|_{W^{1,p'}(\oeps)'}.
\end{align*}

\end{remark}

\begin{corollary}\label{cor:Bogovskii} Let $p \in (1,\infty)$.
For every $f_{\vareps} \in L^p(\oeps)$ we have the following inequality:
\begin{align*}
\|\Beps(f_{\vareps})\|_{L^p(\oeps)} + \vareps \|\nabla \Beps(f_{\vareps})\|_{L^p(\oeps)} &\le C\left( \|f_{\vareps}\|_{L^{\underline{p}}(\oeps)} + \vareps \|\feps\|_{L^p(\oeps)} \right) \le C \|f_{\vareps}\|_{L^p(\oeps)}.
\end{align*}

For $g_{\vareps} \in L^r(\oeps)^3$ with $\div g_{\vareps} \in L^p(\oeps)$ and $g_{\vareps} \cdot \nu = 0$ on $\partial\oeps$ it holds that
\begin{align*}
    \|\Beps(\div g_{\vareps})\|_{L^r(\oeps)} \le C \|g_{\vareps}\|_{L^r(\oeps)}.
\end{align*}
Further, let $n=3$ and $p\in [2,6]$. With $\theta := \frac{3p - 6}{2p} \in [0,1]$ for every $f_{\vareps}\in L^p(\oeps)$ it holds that
\begin{align*}
\vareps \|\Beps(f_{\vareps})\|_{L^p(\oeps)}  \le C \vareps^{1-\theta}\left( \|f_{\vareps}\|_{L^{\frac65}(\oeps)} + \vareps \|f_{\vareps}\|_{L^2(\oeps)} \right).
\end{align*}
\end{corollary}
\begin{proof}
For $p>\frac{n}{n-1}$ the first inequality is a direct consequence of the Poincar\'e inequality in Lemma \ref{lem:Poincare_inequality}, Lemma \ref{lem:Bogovskii_Beps} and Remark \ref{rem:Bogovskii}, as well as the embedding $L^{\tilde{q}}(\oeps) \hookrightarrow W^{1,q}(\oeps)' $ for $\tilde{q}:= \frac{nq}{nq - n -q }$ for $q > \frac{n}{n-1}$. We also used that $\tilde{p'} = \frac{np}{n+p} = \underline{p}$. The case $p \le \frac{n}{n-1}$ can be treated in a similar way, where we notice that $L^q(\oeps) \hookrightarrow W^{1,p'}(\oeps)'$ for every $q>1$.
The inequality for $\Beps(\div g_{\vareps})$ is just $\eqref{ineq:Restriction_operator_div}$.

It remains to show the last inequality. Using the Gagliardo-Nirenberg interpolation and the Poincar\'e inequality from Lemma \ref{lem:Poincare_inequality} (see also $\eqref{ineq:Gagliardo_Nirenberg_ueps}$), we obtain for $f_{\vareps} \in L^p(\oeps)$
\begin{align*}
\|\Beps(f_{\vareps}) 
\|_{L^p(\oeps)} &\le C \left( \|\nabla \Beps (f_{\vareps}) \|_{L^2(\oeps)}^{\theta} \|\Beps(f_{\vareps})\|_{L^2(\oeps)}^{1 -\theta} + \|\Beps(f_{\vareps})\|_{L^1(\oeps)} \right)
\\
&\le C  \vareps^{1-\theta}  \|\nabla \Beps(f_{\vareps})\|_{L^2(\oeps)} .
\end{align*}
Using again Lemma \ref{lem:Bogovskii_Beps} and Remark \ref{rem:Bogovskii} together with the embedding $H^1(\oeps) \hookrightarrow L^6(\oeps)$, we obtain
\begin{align*}
\vareps \|\Beps(f_{\vareps})\|_{L^p(\oeps)} &\le C \vareps^{1-\theta}\left( \|f_{\vareps}\|_{H^1(\oeps)'} + \vareps \|f_{\vareps}\|_{L^2(\oeps)}\right) 
\\
&\le C \vareps^{1-\theta}\left( \|f_{\vareps}\|_{L^{\frac65}(\oeps)} + \vareps \|f_{\vareps}\|_{L^2(\oeps)} \right).
\end{align*}
\end{proof}

\subsection{A convergence result}

The next lemma establishes a strong convergence result for the product of a fixed integrable function and a uniformly bounded sequence that converges strongly to zero in a negative order Sobolev space. Although such results are classical, we provide a proof.

\begin{lemma}\label{lem:strong_convergence_product}
Let $s\in (1,\infty)$ and $f \in L^1((0,T), L^{(s^{\ast})'}(\Omega))$ with $\frac{1}{s^{\ast}} = \frac{1}{s} - \frac13$ and $g_n \in L^{\infty}((0,T)\times \Omega)$ such that
\begin{align*}
\|g_n \|_{L^{\infty}((0,T)\times \Omega)} \le C_g
\end{align*}
for a constant $C_g>0$. Further, we assume that $g_n \rightarrow 0$ in $L^1((0,T),W^{1,s}(\Omega)')$. Then it holds that
\begin{align*}
fg_n \rightarrow 0 \qquad\mbox{strongly in } L^1((0,T),W^{1,s}(\Omega)').
\end{align*}

\end{lemma}
\begin{proof}
By the Sobolev-embedding we have $f\in L^1((0,T),W^{1,s}(\Omega)')$. For given $\delta >0$ there exists $f_{\delta} \in C_0^{\infty}((0,T)\times \Omega)$, such that 
\begin{align*}
\|f - f_{\delta}\|_{L^1((0,T),L^{(s^{\ast})'}(\Omega))} \le \frac{\delta}{2C_g}. 
\end{align*}
Further, we have
\begin{align*}
\|fg_n \|_{W^{1,s}(\Omega)'} \le \|(f - f_{\delta}) g_n \|_{W^{1,s}(\Omega)'} + \|f_{\delta} g_n \|_{W^{1,s}(\Omega)'}.
\end{align*}
For the second term we have 
\begin{align*}
\|f_{\delta} g_n \|_{W^{1,s}(\Omega)'} &= \sup_{\phi \in W^{1,p}(\Omega), \, \|\phi\|_{W^{1,p}(\Omega)} \le 1} \int_{\Omega} f_{\delta} g_n \phi dx 
\\
&\le \sup_{\phi \in W^{1,s}(\Omega), \, \|\phi\|_{W^{1,p}(\Omega)} \le 1} \|g_n \|_{W^{1,p}(\Omega)'} \|f_{\delta}\|_{W^{1,\infty}(\Omega)} 
\\
&\le C_0 \|g_n\|_{W^{1,s}(\Omega)'}
\end{align*}
with a constant $C_0>0$ independent of $n$ (only depending on the choice of $f_{\delta}$). Further, we have
\begin{align*}
\|(f-f_{\delta}) g_n\|_{W^{1,p}(\Omega)'} \le \|g_n\|_{L^{\infty}(\Omega)} \|f - f_{\delta}\|_{L^{(s^{\ast})'}(\Omega)} \le C_g \|f - f_{\delta}\|_{L^{(s^{\ast})'}(\Omega)}.
\end{align*}
Altogether, we have 
\begin{align*}
\|fg_n\|_{L^1((0,T),W^{1,s}(\Omega)')} \le C_0 \|g_n\|_{L^1((0,T),W^{1,s}(\Omega)'} + \frac{\delta}{2} < \delta
\end{align*}
for all $n$ large enough. This gives the desired result.
\end{proof}

\end{appendix}

\section*{Statements and Declarations}

\subsection*{Data Availability}

Data sharing is not applicable to this article as no new data were generated or analyzed
during the current study.

\subsection*{Conflict of interest}

The authors  declare that there is no conflict of interest.

\bibliographystyle{abbrv}
\bibliography{literature}

\end{document}